\documentclass[11pt,a4paper]{article}

\usepackage{natbib}
\usepackage{apalike}
\usepackage{comment}
\usepackage{lipsum}
\usepackage{microtype}

\usepackage[font=small,labelfont=bf]{caption}
\usepackage
[
        a4paper,% other options: a3paper, a5paper, etc
        left=3cm,
        right=3cm,
        top=3cm,
        bottom=3cm,
]
{geometry}

\usepackage{setspace}
\usepackage{graphicx}
\graphicspath{{./figures/}}

\usepackage{multicol,latexsym, array}
\usepackage{textcomp}
\usepackage[ansinew]{inputenc}
\usepackage{amsmath,amssymb,amsthm}
\usepackage[OT1]{fontenc}
\usepackage{hyperref}

\hypersetup{
}

\usepackage{enumitem}
\usepackage{MnSymbol} 
\usepackage{empheq}
\usepackage{cleveref}
\usepackage{rotating}
\usepackage{titletoc}
\usepackage{framed}
\usepackage{multirow,bigdelim}
\usepackage{subfigure}
\usepackage{bbm}
\usepackage[noend]{algpseudocode}
\usepackage{algorithm}
\usepackage{pgf}
\usepackage{pgfplots}
\usepackage{tikz}
\usepackage{shadethm}
\usepackage{thmtools}
\usepackage{mdframed}
\usepackage{lipsum}
\usepackage{footmisc}
\usepackage{fixmath}
\usepackage{multicol}

\usepackage{caption}
\usepackage{xargs}                      % Use more than one optional parameter in a new commands
\usepackage{xcolor}  % Coloured text etc. 
\usepackage[colorinlistoftodos,prependcaption,textsize=tiny]{todonotes}

\definecolor{DarkBlue}{HTML}{2c7f90}
\definecolor{DarkBrown}{HTML}{88632a}
\definecolor{texMidnight}{HTML}{1D1F21}

\providecommand{\keywords}[1]{\textbf{\textit{Keywords: }} #1}

\date{\today}

\newenvironment{ackno}%
    {\paragraph{Acknowledgements }}

\usepackage{caption}
\usepackage{mathrsfs}
\usetikzlibrary{arrows}

\usepackage{xpatch}

\makeatletter
\xpatchcmd{\endmdframed}
  {\aftergroup\endmdf@trivlist\color@endgroup}
  {\endmdf@trivlist\color@endgroup\@doendpe}
  {}{}
\makeatother

\usepackage{nicefrac}
\usepackage{dsfont}

\usepackage{bigints}
\usepackage{mathtools}
\usepackage{fixltx2e}

\usepackage{mathtools}
\usepackage{color}
\usepackage{fancyhdr}
\usepackage{lastpage}

\usepackage{ifthen}

\newcommand{\OT}{\mathrm{OT}}

\newcommand{\R}{\mathbb{R}}
\newcommand{\RR}{\mathbb{R}}

\newcommand{\N}{\mathbb{N}}

\newcommand{\B}{\mathbb{B}}
\newcommand{\GG}{\mathbb{G}}

\newcommand{\XC}{\mathcal{X}}

\newcommand{\YC}{\mathcal{Y}}

\newcommand{\CC}{\mathcal{C}}

\newcommand{\FC}{\mathcal{F}}

\newcommand{\NC}{\mathcal{N}}

\newcommand{\PC}{\mathcal{P}}
\newcommand{\TC}{\mathcal{T}}

\newcommand{\norm}[1]{\left\lVert#1\right\rVert}
\newcommand{\normal}{\mathcal{N}}

\newcommand{\id}{\mathrm{id}}
\newcommand{\Unif}{\mathrm{Unif}}

\newcommand{\dif}{\,\mathrm{d}}

\newcommand{\supp}{\mathrm{supp}}

\newcommand{\EV}[1]{\mathbb{E}\left[#1 \right]}

\newcommand{\Var}{\mathrm{Var}}

\newcommand{\prob}[1]{\mathbb{P}\left(#1\right)}

\newcommand{\interior}[1]{\mathrm{int}\left(#1\right)}

\renewcommand{\DH}{\mathcal{D}^H}

\newcommand{\powerC}{^c}

\newcommand{\ellInf}[1]{\ell^{\infty}\left(#1\right)}

\newcommand{\coloneqq}{:=}

\renewcommand{\phi}{\varphi}

\newcommand{\konvD}{\xrightarrow{\;\;\mathcal{D}\;\;}}

 \newcommand{\konvP}{\xrightarrow{\;\;\mathbb{P}\;\;}}

\renewcommand*{\epsilon}{\varepsilon}

\newcommand{\closure}{{\mathrm{Cl}}}

\newcommand{\BL}[1]{\mathrm{BL}_{1}(#1)}

\DeclareMathOperator*{\argmin}{\mathrm{argmin}}

\theoremstyle{plain}
\newtheorem{theorem}{Theorem}[section]
\newtheorem{corollary}[theorem]{Corollary}
\newtheorem{lemma}[theorem]{Lemma}

\newtheorem{fact}[theorem]{Theorem}

\newcommand{\thissubtheorem}{}
\newcommand{\subtheorem}[1][]{%
  \if\relax\detokenize{#1}\relax
    \def\thissubtheorem{}%
  \else
    \def\thissubtheorem{ (#1)}%
  \fi
  \item
}

\theoremstyle{definition}
\newtheorem{definition}[theorem]{Definition}
\newtheorem{remark}[theorem]{Remark}
\newtheorem{example}[theorem]{Example}

\makeatletter
\newcommand{\mylabel}[2]{#2\def\@currentlabel{#2}\label{#1}}
\makeatother

\numberwithin{equation}{section}

\newcommand{\footremember}[2]{
	\footnote{#2}
	\newcounter{#1}
	\setcounter{#1}{\value{footnote}}
}

\newcommand{\footrecall}[1]{
	\footnotemark[\value{#1}]
}

\providecommand{\keywords}[1]{\textbf{\textit{Keywords}}#1}
\providecommand{\keywordsMSC}[1]{\textbf{\textit{MSC 2020 subject classification}} #1}

\begin{document}
\parindent 0pt

\makeatletter
\def\namedlabel#1#2{\begingroup
    #2%
    \def\@currentlabel{#2}%
    \phantomsection\label{#1}\endgroup
}
\makeatother

\author{Shayan Hundrieser
  \hspace{-0.65em}\footremember{ims}{\scriptsize
    Institute for Mathematical
		Stochastics, University of G\"ottingen,
		Goldschmidtstra{\ss}e 7, 37077 G\"ottingen}%
  \hspace{-0.65em}\footremember{mbexc}{\scriptsize
    Cluster of Excellence "Multiscale Bioimaging: from Molecular Machines to Networks of Excitable Cells" (MBExC),
    University Medical Center,
    Robert-Koch-Stra{\ss}e 40, 37075 G\"ottingen}
	\\
  \footnotesize{\href{mailto:s.hundrieser@math.uni-goettingen.de}{s.hundrieser@math.uni-goettingen.de}}
  \\[2ex]
  Marcel Klatt
  \hspace{-0.65em}\footrecall{ims}%
	\\
  \footnotesize{\href{mailto:thomas.staudt@uni-goettingen.de}{mklatt@mathematik.uni-goettingen.de}}
  \\[2ex]
	Thomas Staudt
  \hspace{-0.65em}\footrecall{ims}%
  \hspace{-0.3em}\footrecall{mbexc}
	\\
  \footnotesize{\href{mailto:thomas.staudt@uni-goettingen.de}{thomas.staudt@uni-goettingen.de}}
  \\[2ex]
	Axel Munk
  \hspace{-0.6em}\footrecall{ims}%
  \hspace{-0.3em}\footrecall{mbexc}%
  \hspace{-0.0em}\footnote{\scriptsize
    Max Planck Institute for Biophysical Chemistry,
    Am Fa{\ss}berg 11, 37077 G\"ottingen}
	\\
  \footnotesize{\href{mailto:munk@math.uni-goettingen.de}{munk@math.uni-goettingen.de}}
}

\title{A Unifying Approach to Distributional Limits for\\  Empirical Optimal Transport}
 
\pagenumbering{arabic}

%Nightmode
%\pagecolor{texMidnight}
%\color{white}

\maketitle
\vspace{-0.8cm}
\begin{abstract}
 We provide a unifying approach to \emph{central limit type theorems} for empirical optimal transport (OT).
In general, the limit distributions are characterized as suprema of Gaussian processes. We explicitly characterize when the limit distribution is centered normal or degenerates to a Dirac measure. Moreover, in contrast to recent contributions on distributional limit laws for empirical OT on Euclidean spaces which require centering around its expectation, the distributional limits obtained here are centered around the \emph{population} quantity, which is well-suited for statistical applications. 

 At the heart of our theory is Kantorovich duality representing OT as a supremum over a function class $\FC_{c}$ for an underlying sufficiently regular cost function $c$. In this regard, OT is considered as a functional defined on $\ellInf{\FC_{c}}$ the Banach space of bounded functionals from $\FC_{c}$ to $\mathbb{R}$ and equipped with uniform norm. We prove the OT functional to be \emph{Hadamard directional differentiable} and conclude distributional convergence via a \emph{functional delta method} that necessitates weak convergence of an underlying empirical process in $\ellInf{\FC_{c}}$. The latter can be dealt with \emph{empirical process theory} and requires $\FC_{c}$ to be a \emph{Donsker} class. We give sufficient conditions depending on the dimension of the ground space, the underlying cost function and the probability measures under consideration to guarantee the Donsker property. Overall, our approach reveals a noteworthy trade-off inherent in central limit theorems for empirical OT: Kantorovich duality requires $\FC_{c}$ to be sufficiently rich, while the empirical processes only converges weakly if $\FC_{c}$ is not too complex.
\end{abstract}
\vspace{0.0cm}

\noindent\keywords{Central limit theorem, optimal transport, Wasserstein distance, regularity theory, empirical processes, bootstrap, Kantorovich potential}
\vspace{0.3cm}

\noindent\keywordsMSC{Primary: 60B12, 60F05, 60G15, 62E20, 62F40; Secondary: 90C08, 90C31}
\vspace{0.3cm}

\setlength{\parindent}{15pt}

\section{Introduction}

Comparing probability distributions is a fundamental task in statistics, probability theory, machine learning, data analysis and related fields. 
From this viewpoint, in addition to longstanding mathematical interest, optimal transport (OT) based metrics have recently gained increasing attention for data analysis as well. A major reason is that OT metrics and related similarity measures not only allow comparing general probability distributions, but can also be designed to respect the metric structure of the underlying ground space. This often results in visually appealing and well interpretable outcomes which together with recent computational progress explains the advancement of OT based data analysis throughout various disciplines, ranging from economics \citep{galichon2016optimal} to statistics \citep{panaretos2019statistical}, machine learning \citep{cuturi18}, signal and image processing \citep{bonneel2011displacement,kolouri2017optimal} and biology \citep{schiebinger2019optimal,tameling2021Colocalization}, among others.

In the following, we consider Polish spaces $\XC$ and $\YC$ and denote the set of probability measures thereon by $\PC(\XC)$ and $\PC(\YC)$, respectively. Given some non-negative measurable cost function $c\colon \XC\times \YC\to \mathbb{R}_+$, the OT cost between $\mu\in\PC(\XC)$ and $\nu\in\PC(\YC)$ is defined as 
\begin{equation}\label{eq:OT}
    \OT_c(\mu,\nu)\coloneqq \inf_{\pi\in \Pi(\mu,\nu)} \int_{\XC\times\YC} c(x,y) \dif\pi(x,y),
\end{equation}
where set $\Pi(\mu, \nu)$ denotes the collection of probability measures on $\XC\times\YC$ such that their marginal distributions coincide with $\mu$ and $\nu$, respectively. Any optimizer $\pi\in \Pi(\mu, \nu)$ for \eqref{eq:OT} is termed OT plan. Essentially, OT in \eqref{eq:OT} comprises the challenge to transform the measure $\mu$ into the measure $\nu$ in a cost optimal way. Under mild assumptions on the cost function, $\OT_c$ in \eqref{eq:OT} enjoys for any pair of measures $\mu\in \PC(\XC), \nu\in \PC(\YC)$ the dual formulation 
\begin{equation}\label{eq:OTdual}
    \OT_c(\mu,\nu)=\sup_{f\in \FC_{c}} \int_{\XC} f(x) \dif\mu(x) + \int_{\YC} f^c(y) \dif\nu(y),
\end{equation}
formally known as \emph{Kantorovich duality}. The function class $\FC_c$ depends on the underlying cost function and $f^c(y)=\inf_{x\in\XC} c(x,y)-f(x)$ is the \emph{$c$-conjugate} of $f\in \FC_{c}$. Any function $f$ attaining the supremum in \eqref{eq:OTdual} is termed \emph{Kantorovich potential} and the set 
\begin{equation}\label{eq:KantorovichPotentials}
	S_c(\mu,\nu)\coloneqq \left\lbrace f\in\FC_{c}\, \mid\, \OT_c(\mu,\nu)=\int_{\XC} f(x) \dif\mu(x) + \int_{\YC} f^c(y) \dif\nu(y)\right\rbrace
\end{equation}
denotes the collection of all Kantorovich potentials. We refer to the monographs by \cite{rachev1998massTheory,rachev1998massApplications}, \cite{vil03, villani2008optimal}, \cite{santambrogio2015optimal} for comprehensive treatment, and to \Cref{sec:GeneralOT} for further details. In a statistical application the probability measures $\mu$ and $\nu$ are estimated from data. For the moment we assume $\nu$ to be known and that we have access to realizations of independent and identically distributed (i.i.d.) random variables $X_1,\ldots,X_n\sim\mu$. The corresponding \emph{empirical measure} $\hat{\mu}_n=\frac{1}{n}\sum_{i=1}^n \delta_{X_i}$ thus serves as a proxy for $\mu$. The \emph{population quantity} $\OT_c(\mu,\nu)$ is then estimated by the empirical OT cost $\OT_c(\hat{\mu}_n,\nu)$, posing questions about its statistical performance.

To this end, distributional limits for the (properly standardized) empirical OT cost are 
fundamental as they capture the asymptotic fluctuation around their population quantities. In the following, we denote such results as \emph{central limits theorems} (CLTs).\footnote{
Historically, the term CLT meant to describe the asymptotic distribution of a \emph{sum} of random variables \citep{leCam86}. Indeed, the empirical OT cost is the sum of costs between random points weighted according to an OT plan.} 
Available results in the literature can be broadly distinguished between the case that $\mu = \nu$ (the null hypothesis in the context of statistical testing) and $\mu \neq \nu$ (corresponding to the alternative when testing $\mu = \nu$). 
A well studied setting in this regard is the $p$-th order \emph{Wasserstein distance}\footnote{To alleviate notation, we write $\OT_p(\mu,\nu)$ for probability measures $\mu,\nu\in\PC(\R^d)$ supported on a Euclidean space and cost function equal to $c(x,y)=\Vert x-y\Vert^p$ for some $p\geq 1$. The Wasserstein distance is then equal to $\smash{\OT_p^{\tiny \nicefrac{1}{p}}(\mu,\nu)}$ and commonly denoted by $W_p(\mu,\nu)$. Analogously, the set of Kantorovich potentials in \eqref{eq:KantorovichPotentials} is denoted by $S_p(\mu,\nu)$ in this case.} $\smash{\OT_p^{1/p}(\mu,\nu)}$ on $\R^d$ that arises by choosing the cost $c(x,y)\coloneqq\Vert x-y\Vert^p$ and probability measures with finite $p$-th moments in \eqref{eq:OT}. 
First analyses have been devoted to the real line $(d=1)$ for which the $p$-th order Wasserstein distance is equal the $L^p$ distance between the quantile functions of the measures. Under the null $\mu = \nu$ and $p = 1$, early contributions by \cite{del1999central} (see also \citealt{mason2016weighted}) provide necessary and sufficient conditions on the probability measures such that the random quantity $\sqrt{n}\OT_1(\hat{\mu}_n,\mu)$ weakly converges towards an integral of a suitable Brownian bridge. A weak limit for $p=2$ is obtained by \cite{del2005asymptotics}. The regime $p \in (1,2)$ was analyzed only recently by \cite{berthet2019weak}. For $p > 2$, CLTs are not immediately available but the work by \cite{bobkov2019one} indicates that similar results can be obtained from general quantile process theory \citep{csorgo1993weighted}. Under the alternative $\mu \neq \nu$ for $p \geq 1$, the random quantity $\sqrt{n}(\OT_p(\hat \mu_n, \nu) - \OT_p(\mu, \nu))$ often asymptotically follows a centered Gaussian distribution \citep{Munk98,del2019central,berthet2019weak,berthet2020central}. Similar in spirit are recent contributions by \cite{Hundrieser2021_Circular} for measures supported on the circle.

While the previous works all benefit from the representation of the OT plan as a quantile coupling when the ground space is totally ordered, the general situation is much more complicated. 
A unifying analysis for discrete metric spaces $\XC=\{x_1,x_2,\ldots\}$ and a metric based cost function $c(x_i,x_j)=d^p(x_i,x_j)$ for $p \geq 1$ is given by \cite{sommerfeld2018} and \cite{tameling18}, which could be seen as a starting point for this paper. For arbitrary measures $\mu,\nu \in \PC(\XC)$ the limiting random variable is specifically characterized applying the functional delta method in terms of a supremum over (infinite dimensional) Gaussian random vectors $\mathbb{G}_\mu \sim \normal(0, \Sigma(\mu))$, namely
\begin{equation}\label{eq:CLTdiscretecase}
    \sqrt{n}\left( \OT_{c}(\hat{\mu}_n,\nu)- \OT_{c}(\mu,\nu)\right) \konvD \sup_{f\in S_c(\mu,\nu)} \left\langle \mathbb{G}_\mu,f\right\rangle,
\end{equation}
where the supremum is taken over $S_c(\mu,\nu)$, the set of optimal Kantorovich potentials in \eqref{eq:KantorovichPotentials}. Parallel and independently to our work, \cite{del2022central} extended this to semi-discrete OT for which \eqref{eq:CLTdiscretecase} remains valid even if the discrete measure $\nu$ is replaced by a general probability measure supported on some Polish space, e.g., absolutely continuous with respect to Lebesgue measure on $\R^d$ (see \Cref{sec:LLspecificinstance}). 

Beyond the countable, one-dimensional and semi-discrete case, the asymptotic distributional behavior for empirical OT becomes much more involved. Already for $d = 2$, precise CLTs remain elusive as highlighted by \cite{ajtai1984optimal} (see also \citealt{talagrand1994matching,bobkov2019simple}) who showed for the uniform distribution $\mu$ on the unit square that $\OT_1(\hat{\mu}_n,\mu)\asymp (\log(n)/n)^{1/2}$ with high probability.
For higher dimensions $d\geq 3$, it is well known that any absolutely continuous measure $\mu$ with compact support on $\mathbb{R}^d$ fulfills $\OT_1(\hat{\mu}_n,\mu)\asymp n^{-1/d}$ \citep{dudley1969, dobric1995asymptotics} which implies $\sqrt{n}\OT_1(\hat{\mu}_n,\mu)$ to diverge. 
Indeed, for general $p \geq 1$ the literature on the convergence of empirical OT is vast and we therefore only give a selective view on recent papers biased towards our main results. Overall, slow convergence rates in the high-dimensional regime seem inevitable as the  Wasserstein distances of any order $p\geq 1$ suffers from the \emph{curse of dimensionality} $\mathbb{E}[\OT_p(\hat{\mu}_n,\mu)]\asymp n^{-p/d}$ whenever $d> 2p$ and $\mu$ is Lebesgue absolutely continuous \citep{boissard2014, fournier2015rate, weed2019sharp}. This demonstrates the scaling rate $\sqrt{n}$ in \eqref{eq:CLTdiscretecase} to be of wrong order for high dimensional (Euclidean) spaces. However, when centering with the \emph{empirical expectation}, concentration results demonstrate the random quantity $\sqrt{n}\,\left(\OT_2(\hat{\mu}_n,\nu)- \mathbb{E}[\OT_2(\hat{\mu}_n,\nu)]\right)$ to be tight \citep{weed2019sharp,Chizat2020}. A fundamental step further has been taken by \cite{delbarrio2019} who obtain for probability measures with $4+\delta$ finite moments ($\delta>0$) and a positive density in the interior of their convex support that 
\begin{equation}\label{eq:delbarrioCLT}
    \sqrt{n}\left(\OT_2(\hat{\mu}_n,\nu)- \mathbb{E}[\OT_2(\hat{\mu}_n,\nu)]\right) \konvD Z\sim \NC\big(0,\mathrm{Var}_{X\sim\mu}[f(X)]\big).
\end{equation}
Here and in the following, $\NC(\mu, \sigma^2)$ denotes the normal distribution with mean $\mu$ and variance $\sigma^2$. 
The asymptotic variance $\mathrm{Var}_{X\sim \mu}[f(X)]$ in \eqref{eq:delbarrioCLT} is equal to the variance of the random variable $f(X)$ with $X\sim\mu$ and $f$ the \emph{unique} Kantorovich potential for \eqref{eq:OTdual}. Under the null $\mu=\nu$, the asymptotic variance is equal to zero and the CLT in \eqref{eq:delbarrioCLT} degenerates,  in contrast to the alternative $\mu \neq \nu$ that usually leads to a non-degenerate normal limit law. The approach by \cite{delbarrio2019} relies on approximating the empirical OT cost via \eqref{eq:OTdual} as a linear functional involving a unique Kantorovich potential for which a CLT immediately follows. Based on the Efron-Stein variance inequality this functional is shown to serve as a good $L^2$-approximizer for the random quantity $\sqrt{n}\,\left(\OT_2(\hat{\mu}_n,\nu)- \mathbb{E}[\OT_2(\hat{\mu}_n,\nu)]\right)$ which yields the conclusion. These results have recently been extended by \cite{delBarrio2021GeneralCosts} to more general convex cost functions, including a CLT similar to \eqref{eq:delbarrioCLT} for $\OT_p$ under $p >1$ provided the probability measures have moments of order $2p$. Notably, the regularity conditions on the measures impose the Kantorovich potential to be unique and it remains unclear how to generalize their approach under non-uniqueness. More crucially, the statement requires centering around the empirical expectation which hinders its immediate use for statistical applications. It might be tempting to replace $\mathbb{E}[\OT_2(\hat{\mu}_n,\nu)]$ by its population counterpart $\OT_2(\mu,\nu)$ in \eqref{eq:delbarrioCLT}. On the real line $(d=1)$ and under sufficient regularity assumptions, this is possible \citep{del2019central}. However, it remains a delicate issue for $d\geq 2$. By an observation of \citet[Proof of Proposition 21]{manole2021sharp}, if $\mu$ and $\nu$ are uniform measures on two different balls of equal radius the bias is lower bounded by $\mathbb{E}[\OT_2(\hat{\mu}_n,\nu)]-\OT_2(\mu,\nu)\gtrsim n^{-2/d}$. This demonstrates the replacement of the centering with the population quantity in \eqref{eq:delbarrioCLT} to be invalid for $d\geq 5$ (the special case $d=4$ is further addressed in \Cref{ex:Dgtr4}).  
  Nevertheless, employing a different estimator may allow under additional assumptions for faster convergence rates of the bias in high dimensions. Indeed, for probability measures $\mu$,~$\nu$ on the unit cube $[0,1]^d$ with sufficiently smooth densities \cite{Manole2021_Plugin} propose a suitable wavelet estimator $\tilde \mu_n$ and prove the bias to be of order $|\EV{\OT_2(\tilde\mu_n, \nu)} - \OT_2(\mu, \nu)|=o(n^{-1/2})$. Combined with a strategy as outlined by \cite{delbarrio2019}, the random quantity $\sqrt{n}(\OT_2(\tilde\mu_n, \nu) - \OT_2(\mu, \nu))$ is shown to asymptotically follow a centered Gaussian distribution analogous to \eqref{eq:delbarrioCLT}, which also degenerates under the null $\mu=\nu$. Despite this being an interesting result, CLTs for empirical OT costs based on general cost functions and centered around the population quantity remain largely open. 

In this work, we provide a unifying approach to obtain CLTs for empirical OT that include certain aforementioned settings but go far beyond. In particular, our approach does not rely on discrete spaces, neither on explicit formulas involving quantiles $(d=1)$ nor unique Kantorovich potentials $(d\geq 2)$. Our limit laws are reminiscent of the discrete case \eqref{eq:CLTdiscretecase} but hold for considerably more general cost functions and probability measures. The limiting random variables are characterized as suprema of Gaussian processes $\mathbb{G}_\mu$ in $\ellInf{\FC_{c}}$ and indexed in Kantorovich potentials from $S_c(\mu,\nu)$ in \eqref{eq:KantorovichPotentials}. The CLTs are centered around the \emph{population quantity}. Our main result (\Cref{thm:GeneralCLT}) states under certain assumptions that 
\begin{equation}\label{eq:introresultOTmu}
    \sqrt{n}\left( \OT_c(\hat \mu_n,\nu) - \OT_c(\mu,\nu) \right) \konvD  \sup_{ f\in S_c(\mu, \nu)} \mathbb{G}_\mu(f).
\end{equation}
The distributional limit law \eqref{eq:introresultOTmu} is valid whenever $\FC_c$ in \eqref{eq:OTFunctionclassLarge} (below) is $\mu$-Donsker and  
\begin{enumerate}[label=\textbf{(C)}]
    \item \label{ass:CostsContinuousAndBounded} The cost $c\colon \XC\times \YC \rightarrow \R_+$ is continuous and bounded by $\norm{c}_\infty\coloneqq\sup_{x,y}\vert c(x,y) \vert<\infty$.
\end{enumerate}
holds true combined with one of the following two assumptions:
\begin{enumerate}[label=\textbf{(S\arabic*)}]
    \item\label{ass:EquicontinuityLocallyCompact} The spaces $\XC$ and $\YC$ are locally compact and $\{c(\cdot,y)\mid y\in \YC\}$ and $\{c(x,\cdot)\mid x\in \XC\}$ are equicontinuous\footref{fn:equicontinuity} on $\XC$ and $\YC$, respectively.
    \item\label{ass:EquicontinuityCompact} The space $\XC$ is compact and $\{c(\cdot,y)\mid y\in \YC\}$ is equicontinuous\footnote{\label{fn:equicontinuity} Equicontinuity refers to a common modulus of continuity for the function classes $\big\{ c(\cdot, y)\,|\,y\in\YC\big\}$ and $\big\{c(x, \cdot)\,|\,x\in\XC\big\}$ with respect to a continuous metric on $\XC$ and $\YC$.} on $\XC$.
\end{enumerate}
In particular, under Assumption \ref{ass:CostsContinuousAndBounded} the function class $\FC_c$ used in \eqref{eq:OTdual} can be chosen as a uniformly bounded class of $c$-concave functions on $\XC$ \cite[Remark 1.13]{vil03}
\begin{align}\label{eq:OTFunctionclassLarge}
    \FC_{c}\coloneqq\left\lbrace f\colon \XC\to \mathbb{R} \, \mid\, \exists\, g\colon\YC\to\mathbb{R},\, -\lVert c\rVert_\infty \leq g \leq 0,\, f=\inf_{y\in\YC} c(\cdot,y)-g(y)\right\rbrace.
\end{align}
Moreover, $\mathbb{G}_\mu$ in \eqref{eq:introresultOTmu} is a tight centered Gaussian process and represents the weak limit of the empirical process $\sqrt{n}(\hat{\mu}_n-\mu)$ in $\ellInf{\FC_c}$, the Banach space of bounded functionals from $\FC_c$ to $\mathbb{R}$ equipped with uniform norm. The covariance of $\mathbb{G}_\mu$ specifically depends on $\mu$ and is equal to
\begin{equation}\label{eq:introcovariance}
    \mathbb{E}\left[ \mathbb{G}_\mu(f_1)\mathbb{G}_\mu(f_2)\right] = \int_{\XC} f_1(x) f_2(x) \dif \mu(x) -\int_{\XC} f_1(x)\dif \mu(x) \int_{\XC} f_2(x) \dif\mu(x)
\end{equation}
for two functions $f_1,f_2\in\FC_c$. If instead $\nu$ is estimated by its empirical version, then under the same conditions on the cost and the spaces, i.e.,  Assumption \ref{ass:CostsContinuousAndBounded} combined with  \ref{ass:EquicontinuityLocallyCompact} or \ref{ass:EquicontinuityCompact}, and if $\FC_c^c$, the set of all $c$-conjugate functions from $\FC_c$, is $\nu$-Donsker, our CLT reads as  
\begin{equation}\label{eq:introresultOTnu}
    \sqrt{n}\left( \OT_c(\mu,\hat{\nu}_n) - \OT_c(\mu,\nu) \right) \konvD  \sup_{ f\in S_c(\mu, \nu)} \mathbb{G}_\nu(f^c),
\end{equation}
where $\mathbb{G}_\nu$ is a centered Gaussian process in the Banach space $\ellInf{\FC^c_c}$ with similar covariance as in \eqref{eq:introcovariance} corresponding to the weak limit of the empirical process $\sqrt{n}(\hat{\nu}_n-\nu)$.

Our proof technique relies on Kantorovich duality that represents the $\OT_c(\cdot,\cdot)$ cost as a functional from $\ellInf{\FC_c}\times \ellInf{\FC_c\powerC}$ to $\RR$. We take upon this approach in \Cref{sec:GeneralOT} and prove that $\OT_c(\cdot,\cdot)$ is Hadamard directionally differentiable. An application of a general functional delta method \citep{dumbgen1993nondifferentiable,Roemisch04} allows to conclude our main results on CLTs for empirical OT. This necessitates the empirical process $\sqrt{n}(\hat{\mu}_n-\mu)$ to converge weakly to some tight random element $\mathbb{G}_\mu$ in $\ellInf{\FC_c}$. The latter can be dealt with \emph{empirical process theory} and requires the function class $\FC_c$ to be $\mu$-\emph{Donsker}. Our results naturally extend to the two sample case where both probability measures are estimated by their empirical versions, simultaneously. We further characterize asymptotic \emph{normality}, i.e., when the supremum in \eqref{eq:introresultOTmu} and \eqref{eq:introresultOTnu} is over a singleton (\Cref{thm:NormalLimits} in \Cref{sec:NormalLimitLaws}) and \emph{degeneracy} (\Cref{thm:DegenerateOTLimits} in \Cref{sec:DegenerateLimitLaws}), i.e., when the limit law is equal to a Dirac measure which results from unique and trivial (almost surely constant) Kantorovich potentials, respectively. We emphasize that even if $\mu=\nu$, the CLTs might be non-degenerate, e.g., if $\mu$ has disconnected support, as for the discrete case in \eqref{eq:CLTdiscretecase}, where limit laws usually do not degenerate to a Dirac at zero \citep{tameling18}.

In light of the general CLT statement in \eqref{eq:introresultOTmu}, we discuss concrete settings for which the assumptions on the cost are satisfied and $\FC_c$ is $\mu$-Donsker (\Cref{sec:LLspecificinstance}). The latter manifests itself in the complexity of $\FC_c$ measured in terms of covering numbers \citep{van1996weak} as well as tail conditions on the measure $\mu$. The covering numbers depend on properties of the cost function and the underlying dimension of the ground space (see also \citealt{gangbo1996geometry,Chizat2020,hundrieser2021convergence}). The simplest case appears on finite spaces where $\FC_c$ is even \emph{universal} Donsker. On countable discrete spaces it requires the popular \emph{Borisov-Dudley-Durst} summability condition on the measure $\mu$. This is in line with the aforementioned results by \cite{sommerfeld2018} and \cite{tameling18} (\Cref{cor:countableCLTs}). Beyond discrete spaces, we employ well-known results in empirical process theory. More precisely and tailored to Euclidean spaces $\R^d$ with $d\leq 3$, if the cost satisfies certain regularity conditions, then $\FC_c$ is $\mu$-Donsker provided the measure $\mu\in\PC(\R^d)$ satisfies 
\begin{align*}
     \sum_{k \in \mathbb{Z}^d} \sqrt{\mu\left([k,k+1)\right)} < \infty.
\end{align*}
This implies CLTs for the empirical OT cost on the real line with general cost functions (\Cref{cor:CLT_RealLine}) but also novel statements for dimension $d=2,3$ (\Cref{cor:CLT_EuclSpace}). Moreover, in \Cref{ex:Dgtr4} we show for $\mu = \Unif([0,1]^d)$ that the $\mu$-Donsker property of $\FC_c$ is already violated under smooth costs $c(x,y) = \norm{x-y}^2$ for $d \geq 4$ and that CLTs as in \eqref{eq:introresultOTmu} cannot hold for $d\geq 5$ (for $d = 4$ such CLTs are still possible though and we give an example). In view of this observation, our approach is exhaustive in terms of the dimension. Nevertheless, provided only one of the probability measures is supported on low-dimensional compact smooth submanifold of $\RR^d$ we highlight that the Donsker property is still fulfilled and that the CLTs in \eqref{eq:introresultOTmu} remain valid  (\Cref{cor:CLT_LCA}). This is related to recent findings by \cite{hundrieser2021convergence} discovering the \emph{lower complexity adaptation} phenomenon of empirical OT which states that the convergence rate of the empirical OT cost under different population measures adapts to the measure with lower-dimensional support. Moreover, based on this observation, we also derive a CLT for the empirical OT cost in the semi-discrete framework (\Cref{thm:Semidiscrete}).
 We postpone further technical proofs and auxiliary results on Kantorovich potentials to \Cref{sec:Proofs}.

\paragraph*{Notation}
The set of non-negative real numbers is $\R_+$. For $a,b\in\mathbb{R}$ the inequality $a\lesssim_{\kappa} b$ means that $a$ is larger than $b$ up to a constant depending on $\kappa$. In certain instances we omit the $\kappa$.  If  $a \lesssim b \lesssim a$, we write $a \asymp b$. 
The class of real-valued continuous functions defined on a metric space $\XC$ is denoted by $C(\XC)$. A real-valued function $f$ defined on some convex subset $A\subset \XC$ is $\lambda$-semi-concave if there exists some constant $\lambda>0$ such that $f(x)-\lambda\Vert x\Vert^2_2$ is concave. 
Furthermore, $f$ is said to be $(\alpha,L)$-H\"older continuous if there exist positive constants $\alpha\in(0,1]$ and $L>0$ such that $\vert f(x)-f(x^\prime)\vert \leq L\Vert x-x^\prime \Vert^\alpha$ for all $x,x^\prime$ in the domain of $f$. If $\alpha=1$, then $f$ is $L$-Lipschitz. For a function class $\FC$ on $\XC$ denote by $\norm{f-g}_\infty$ the uniform norm $\norm{f}_\infty=\sup_{x\in\XC}\vert f(x) \vert$. Let $\ellInf\FC$ be the Banach space of real-valued bounded functionals on $\FC$ with respect to uniform norm $\Vert \phi \Vert_\FC\coloneqq \sup_{f\in\FC}\vert \phi(f)\vert$. For $(\mathcal{F},d)$ a subset of some metric space with metric $d$ and $\epsilon>0$, the covering number $\NC(\epsilon,\FC,d)$ is the minimal number of balls $\left\lbrace g\mid\, d(f,g)<\epsilon \right\rbrace$ of radius $\epsilon$ such that their union contains $\FC$. The metric entropy of $\FC$ is the logarithm of the covering number $\log\big(\NC(\epsilon,\FC,d)\big)$.

\section{Central Limit Theorems for Empirical Optimal Transport}\label{sec:GeneralOT}

Throughout this section, we consider Polish spaces $\XC$ and $\YC$ and a non-negative cost function $c\colon \XC\times \YC\rightarrow \R_+$ such that Assumption \ref{ass:CostsContinuousAndBounded} is fulfilled. Under this assumption the OT cost in \eqref{eq:OT} is finite for any two probability measures $\mu\in \PC(\XC)$, $\nu \in \PC(\YC)$ and enjoys Kantorovich duality \eqref{eq:OTdual}. The set $\FC_{c}$, over which the dual is optimized, is the collection of uniformly bounded \emph{$c$-concave} functions on $\XC$ defined as in \eqref{eq:OTFunctionclassLarge} (for details see \citealt[Remark 1.13]{vil03} and the \nameref{sec:Proofs}). The supremum in \eqref{eq:OTdual} over $\FC_{c}$ is attained, i.e., there exists a Kantorovich potential $f\in\FC_{c}$, where we recall the set $S_c(\mu,\nu)$ in \eqref{eq:KantorovichPotentials} of all Kantorovich potentials. Notably, as the cost function is continuous, any function $f\in \FC_{c}$ and its $c$-conjugate $f\powerC$ are upper semi-continuous and thus measurable on the Polish spaces $\XC$ and $\YC$, respectively.\\
More general structural properties for $c$-concave functions $\FC_c$ and its $c$-conjugate class $\FC_c\powerC \coloneqq \{f\powerC \mid f \in \FC_c\}$ are intrinsically linked to the cost function and the underlying Polish space. Hence, to guarantee (minimal) regularity properties of the set of Kantorovich potentials we impose Assumption \ref{ass:EquicontinuityLocallyCompact} or \ref{ass:EquicontinuityCompact}. Under either of these conditions, the Kantorovich potentials $S_c(\mu, \nu)$ and $S_c^c(\mu, \nu)$ are continuous as well (\Cref{lem:ContinuityOfKantPotentials}). Moreover, under compactness or local compactness of the spaces $\XC$ and $\YC$ as well as the equicontinuity condition of the cost function, the classes $\FC_c$ and $\FC_c^c$ are by the \emph{Arzel\`a-Ascoli} theorem relatively compact with respect to uniform convergence on compact sets (see proof of \Cref{thm:GeneralCLT} and in particular Step 3).

Examples of locally compact spaces are Euclidean spaces but also encompass finite and countable spaces equipped with discrete topology. Assumptions \ref{ass:CostsContinuousAndBounded} and \ref{ass:EquicontinuityLocallyCompact} are then fulfilled for the Euclidean case $\XC = \YC = \R^d$ if $c(x,y) = h(x-y)$ for some bounded Lipschitz function $h\colon \R^d\rightarrow \R_+$ and for discrete spaces $\XC, \YC$ if the cost function is uniformly bounded. Furthermore and for compact subsets $\XC,\YC\subset \R^d$, Assumptions \ref{ass:CostsContinuousAndBounded} and \ref{ass:EquicontinuityCompact} are fulfilled for cost $c(x,y) = h(x-y)$ for some (possibly unbounded) Lipschitz function $h\colon  \R^d\rightarrow \R_+$. In particular, this encompasses costs of the form $c(x,y)= \norm{x-y}^p$ with $p \geq 1$ inherent in the popular Wasserstein distance. More examples in this regard are detailed in \Cref{sec:LLspecificinstance}.

\subsection{Main Result}

In the following, the empirical process $\sqrt{n}(\hat \mu_n - \mu)$ is considered as a random element in the Banach space $\ellInf{\FC_{c}}$ equipped with uniform norm. 

\begin{definition}[\citealt{van1996weak}] \label{def:Donsker}
    A function class $\FC$ of measurable, pointwise bounded functions on $\XC$ is called \emph{$\mu$-Donsker}, if the empirical process $\sqrt{n}(\hat \mu_n - \mu)$ weakly converges to a tight random element $\mathbb{G}_\mu$ in $\ell^\infty(\FC)$. Furthermore, $\FC$ is \emph{universal Donsker} if for any probability measure $\mu \in \PC(\XC)$ the function class $\FC$ is $\mu$-Donsker. 
\end{definition}

If $\FC_{c}$ is $\mu$-Donsker, the weak limit $\mathbb{G}_\mu$ is a mean-zero Gaussian process indexed over the function class $\FC_{c}$ with covariance for $f_1,f_2\in\FC_{c}$ as in \eqref{eq:introcovariance}. We now state our main result on the CLTs for the empirical OT cost.

\begin{theorem}\label{thm:GeneralCLT}
Consider Polish spaces $\XC$ and $\YC$ and a cost function $c\colon \XC\times \YC\rightarrow \R_+$ satisfying  Assumption \ref{ass:CostsContinuousAndBounded} combined with  \ref{ass:EquicontinuityLocallyCompact} or \ref{ass:EquicontinuityCompact}. For two probability measures $\mu\in \PC(\XC)$, $\nu\in \PC(\YC)$, i.i.d. random variables $X_1, \dots, X_n\sim \mu$ and independent to that i.i.d. random variables $Y_1, \dots, Y_m\sim \nu$ with $n, m\in \N$ denote their respective empirical measures by $\hat \mu_n \coloneqq \frac{1}{n}\sum_{i=1}^n \delta_{X_i}$ and $\hat \nu_m \coloneqq \frac{1}{m}\sum_{i=1}^m \delta_{Y_i}$. 
	\begin{subequations}\label{eq:GeneralCLT}
	\begin{enumerate}
	\item[$(i)$] (One-sample from $\mu$) Suppose that $\FC_{c}$ is $\mu$-Donsker. Then, for $n\rightarrow \infty$,
	    \begin{equation}\label{eq:GeneralCLT_OneSampleMu}
            \sqrt{n}\left(\OT_c(\hat \mu_n, \nu) - \OT_c(\mu, \nu)\right) \konvD \sup_{f\in S_c(\mu, \nu)} \mathbb{G}_\mu(f).
        \end{equation}
\item[$(ii)$] (One-sample from $\nu$) Suppose that $\FC_{c}\powerC$ is $\nu$-Donsker. Then, for $m\rightarrow \infty$,
        \begin{equation}\label{eq:GeneralCLT_OneSampleNu}
          \sqrt{m}\left(\OT_c(\mu, \hat\nu_m) - \OT_c(\mu, \nu)\right) \konvD \sup_{f\in S_c(\mu, \nu)} \mathbb{G}_\nu(f\powerC).
        \end{equation}
	\item[$(iii)$] (Two-sample) Suppose that $\FC_{c}$ is $\mu$-Donsker and that $\FC_{c}\powerC$ is $\nu$-Donsker. Then, for $n,m \rightarrow \infty$ with $m/(n+m)\rightarrow \delta\in (0,1)$,
	\begin{multline}
	 	\label{eq:GeneralCLT_TwoSample}
         \qquad  \sqrt{\frac{nm}{n+m}}\left(\OT_c(\hat \mu_n, \hat\nu_m) - \OT_c(\mu, \nu)\right)  \\
           \konvD\sup_{f\in S_c(\mu, \nu)} \left(\sqrt{\delta}\mathbb{G}_\mu(f) + \sqrt{1-\delta}\mathbb{G}_\nu(f\powerC)\right).\qquad
	\end{multline}
	\end{enumerate}
	\end{subequations}
\end{theorem}

\begin{proof}
Our proof is based on the Hadamard directional differentiability of the OT cost on the set of probability measures $\PC(\XC)\times\PC(\YC)$ with respect to the topology induced by $\ellInf{\FC_{c}}\times \ellInf{\FC_{c}\powerC}$ (see the subsequent \Cref{thm:OTHadamardDiff}).  Under settings $(i)$ and $(ii)$ 
it holds for $n \rightarrow \infty$ and $m \rightarrow \infty$ that 
\begin{equation*}
    \sqrt{n}(\hat \mu_n - \mu)\konvD \mathbb{G}_\mu \text{ in } \ellInf{\FC_{c}}, \quad \sqrt{m}(\hat \nu_m - \nu)\konvD \mathbb{G}_\nu \text{ in } \ellInf{\FC_{c}\powerC},
\end{equation*}
respectively. For setting $(iii)$ it holds by \cite[Example 1.4.6]{van1996weak} that the empirical processes converge jointly
\begin{equation*}
    \sqrt{\frac{nm}{n+m}}\big(\hat \mu_n - \mu, \hat \nu_m - \nu\big	)\konvD \left(\sqrt{\delta}\mathbb{G}_\mu, \sqrt{1-\delta}\mathbb{G}_\nu\right) \text{ in } \ellInf{\FC_{c}}\times\ellInf{\FC_{c}\powerC}
\end{equation*}
with $m/(n+m)\to\delta\in(0,1)$. Moreover, \Cref{thm:OTHadamardDiff} below asserts that $\OT_c \colon \PC(\XC)\times \PC(\YC) \rightarrow \mathbb{R}$ is  Hadamard directionally differentiable at $(\mu,\nu)$ with respect to $\ellInf{\FC_{c}}\times \ellInf{\FC_{c}\powerC}$ tangentially to $\PC(\tilde \XC)\times \PC(\tilde \YC)$  with $\tilde \XC =\supp(\mu)$ and $\tilde \YC =\supp(\nu)$ and $\TC_{\mu}\big(\PC(\tilde \XC)\big)=\closure{\big\{ \frac{\mu^\prime-\mu}{t} \text{ for } t>0, \mu^\prime\in \PC(\tilde\XC)\big\}}$ and analogously for $\TC_{\nu}\big(\PC(\tilde \YC)\big)$. Portmanteau's Theorem for closed sets \cite[Theorem 1.3.4 (iii)]{van1996weak} proves that 
\begin{equation}\label{eq:PortmanteauArgument}
    \mathbb{P}\left(\mathbb{G}_\mu \in \TC_{\mu}(\PC(\tilde \XC)) \right) \geq \limsup_{n\to\infty} \prob{\sqrt{n}(\hat \mu_n - \mu) \in \TC_{\mu}(\PC(\tilde \XC)) } = 1
\end{equation}
since $\supp(\mu_n) \subseteq \supp(\mu)$, and analogously for $\mathbb{G}_\nu$. Then, the statements in \Cref{thm:GeneralCLT} follow from the functional delta method \cite[Theorem 1]{Roemisch04}. 
\end{proof}

We investigate specific settings that yield novel CLTs in \Cref{sec:LLspecificinstance}. At this stage, we like to highlight that \Cref{thm:GeneralCLT} links CLTs for the empirical OT cost to the study of the function classes $\FC_c$ and $\FC_c^c$ being Donsker, respectively.
\begin{remark}
Even if Assumptions \ref{ass:EquicontinuityLocallyCompact} and \ref{ass:EquicontinuityCompact} fail to hold, the first step in the proof of \Cref{thm:OTHadamardDiff} still asserts  Hadamard directional differentiability of the OT cost. However, in this case the derivative rather depends on the set of $\epsilon$-approximate optimizer $S_c(\mu, \nu, \epsilon)\coloneqq \left\{ f \in \FC_{c} \,\mid\, \mu(f)  +\nu(f\powerC) \geq  \OT_c(\mu, \nu) - \epsilon\right\}$ instead of the set of Kantorovich potentials $S_c(\mu, \nu)$. Hence, employing a functional delta method \citep[Theorem 1]{Roemisch04}, we obtain that, as long as $\FC_c$ is $\mu$-Donsker, a CLT of the form  
\begin{equation*}
    \sqrt{n}\Big(\OT_c(\hat \mu_n, \nu) - \OT_c(\mu, \nu)\Big) \konvD \lim_{\epsilon \searrow 0}\sup_{f \in S_c(\mu, \nu, \epsilon)} \mathbb{G}_\mu(f)
\end{equation*}
is valid. Assumptions \ref{ass:EquicontinuityLocallyCompact} and \ref{ass:EquicontinuityCompact} serve to simplify the limit distribution obtained in \Cref{thm:GeneralCLT}. For more details we refer to \Cref{subsec:HadamardDiff}.
\end{remark}

\begin{remark}[Bootstrap consistency] It is well-known that the naive $n$-out-of-$n$ bootstrap fails to be consistent if the functional is not linearly Hadamard differentiable \citep{dumbgen1993nondifferentiable,fang2019}. Instead, Hadamard directional differentiability \citet[Proposition 2]{dumbgen1993nondifferentiable} asserts consistency of the $k$-out-of-$n$ bootstrap for $k = o(n)$ which has immediate consequences for the approximation of quantiles for the empirical OT cost. We formalize this principle exemplary for the one-sample case. For $n, k \in \N$, consider i.i.d. samples $X_1, \dots, X_n\sim \mu$ with empirical measure $\hat \mu_n= \frac{1}{n} \sum_{i = 1}^{n} \delta_{X_i}$ and consider i.i.d. bootstrap samples $X_1^*, \dots, X_k^* \sim \hat \mu_n$ with corresponding (bootstrap) empirical measure $\hat \mu_{n,k}^*= \frac{1}{k} \sum_{i = 1}^{k} \delta_{X_i^*}$. Then, it follows for $n,k \rightarrow\infty $ with $k = o(n)$ that 
 \begin{align*}
      \sup_{h \in \BL{\R}} \Big|  &\EV{h\left( \sqrt{k}(\OT_c(\hat \mu_{n,k}^*, \nu) - \OT_c(\hat \mu_n, \nu) ) \right) \Big| X_1, \dots, X_n } \\
     &- \EV{h\left( \sqrt{n}(\OT_c(\hat \mu_{n}, \nu) - \OT_c(\mu, \nu) ) \right) } \Big| \konvP 0.
 \end{align*}
 Herein, $\konvP$ denotes convergence in outer probability \citep{van1996weak} and $\BL{\R}$  is the set of real-valued functions on $\R$ that are absolutely bounded by one and Lipschitz with modulus one.
\end{remark}

\subsection{Hadamard Directional Differentiability}\label{subsec:HadamardDiff}

Based on Kantorovich duality \eqref{eq:OTdual}, we consider the OT cost as an optimization problem over the set of $c$-concave functions mapping from the set of probability measures $\PC(\XC)\times \PC(\YC)$ to $\RR$. % 
For a pair of probability measures $\mu\in \PC(\XC), \nu\in\PC(\YC)$, we prove Hadamard directional differentiability of the OT cost at $(\mu, \nu)$  with respect to the Banach spaces $\ellInf{\FC_c}\times \ellInf{\FC_c\powerC}$ equipped with uniform norm. For a definition of this notion of differentiability, we refer to \cite{Roemisch04}.  
Moreover, since the support of empirical measures is always contained in the support of the underlying measure, we focus in the following only on differentiability \emph{tangentially} to the set of probability measures whose support is contained in the support of $\mu$ and $\nu$, respectively. This means that only those perturbations for $\mu$ and $\nu$ are considered which do not lead to an enlargement of the support.

\begin{theorem}[Hadamard directional differentiability of OT cost]\label{thm:OTHadamardDiff}
Suppose that for two Polish spaces $\XC$, $\YC$ and the cost function $c\colon \XC\times \YC\rightarrow \R_+$ the Assumption \ref{ass:CostsContinuousAndBounded} combined with \ref{ass:EquicontinuityLocallyCompact} or \ref{ass:EquicontinuityCompact} are satisfied. Consider the function class $\FC_{c}$ in \eqref{eq:OTFunctionclassLarge} and two probability measures $\mu\in \PC(\XC), \nu \in\PC(\YC)$ with their respective support $\tilde \XC\coloneqq \supp(\mu)$ and $\tilde \YC \coloneqq \supp(\nu)$. Then, the OT cost functional 
\begin{equation}\label{eq:OTFunctionalHadDiff} 
\OT_c\colon \PC(\XC)\times \PC(\YC)\subseteq \ell^\infty(\FC_{c})\times  \ell^\infty(\FC\powerC_{c}) \rightarrow \R,\quad (\mu, \nu) \mapsto \sup_{f \in \FC_{c}} \left(\mu(f)  +\nu(f\powerC)\right)
\end{equation}
is  Hadamard directionally differentiable at $(\mu, \nu)$ tangentially to the set of probability measures $ \PC(\tilde \XC)\times \PC(\tilde \YC)$. The derivative is equal to
\begin{align*}
     \DH_{|(\mu, \nu)}\OT_c\colon \TC_{\mu}\big(\PC(\tilde \XC)\big)\times \TC_{\nu}\big(\PC(\tilde\YC)\big)\rightarrow \R,\quad (\Delta_\mu,\Delta_\nu)  \mapsto \sup_{f \in S_c(\mu, \nu)}  \left( \Delta_{\mu}(f) +  \Delta_{\nu}(f\powerC) \right). 
\end{align*}
Herein, the contingent Bouligand cone $\TC_{\mu}\big(\PC(\tilde \XC)\big)$ to $\PC(\tilde \XC)$ at $\mu$ is given by the topological closure $\closure{\big\lbrace \frac{\mu'-\mu}{t} \text{ for } t>0,\, \mu'\in \PC(\tilde \XC)\big\rbrace}\subseteq \ellInf{\FC_c}$ and analogously for $\TC_{\nu}\big(\PC(\tilde \YC)\big)$. 
\end{theorem}

\begin{proof}

The proof is inspired by \citet[Proposition 1]{Roemisch04} and \citet[Theorem 2.1 and Corollary 2.2]{Carcamo20}. Compared to their setting, the proof here is specifically tailored to the OT cost in \eqref{eq:OTdual} and exploits properties of the function class $\FC_c$. We divide the proof into four steps. The first two essentially prove the Hadamard directional differentiability and simplify the representation of the derivative, whereas the last two are concerned with the convergence results for  Kantorovich potentials.\\

\textit{Step 1. Hadamard directional differentiability.} Let $(t_n)_{n \in \N}$ be a positive sequence with $t_n \searrow 0$ and take sequences $(\Delta_{\mu,n},\Delta_{\nu, n}) \in \ell^\infty(\FC_{c})\times \ell^\infty(\FC_{c}\powerC)$ such that for all $n \in \N$ holds $$\mu_n \coloneqq \mu + t_n \Delta_{\mu,n}\in \PC(\tilde \XC), \quad \nu_n \coloneqq \nu + t_n \Delta_{\nu,n}\in \PC(\tilde \YC),$$
	with $(\Delta_{\mu,n},\Delta_{\nu, n}) \rightarrow (\Delta_\mu, \Delta_\nu)\in \TC_{\mu}\big(\PC(\tilde \XC)\big)\times \TC_{\nu}\big(\PC(\tilde \YC)\big)$ for $n \rightarrow \infty$ in the space $\ell^\infty(\FC_{c})\times  \ell^\infty(\FC_{c}\powerC)$. The representation of the Bouligand cone as a topological closure follows by an observation of \cite{Roemisch04} since $\PC(\tilde \XC)$ and $\PC(\tilde \YC)$ are convex sets. For $\mu, \mu_n$ and $\nu, \nu_n$ considered as bounded functionals on $\FC_c$ and $\FC_c\powerC$, respectively, it follows along the lines of \citet[Proposition 1]{Roemisch04} that the OT cost is Hadamard directionally differentiable with 
$$ 
\DH_{|(\mu, \nu)}(\Delta_\mu,\Delta_\nu) = \lim_{n\rightarrow \infty} \frac{1}{t_n}\left(\OT_c(\mu_n, \nu_n) - \OT_c(\mu, \nu) \right)=
 % = 
\lim_{\epsilon \searrow 0}\sup_{f \in S_c(\mu, \nu,\epsilon) } ( \Delta_{\mu}(f) + \Delta_{\nu}(f\powerC)),
$$
 where $S_c(\mu, \nu,\epsilon)$ denotes the set of $\epsilon$-approximizers 
\begin{align*}
    S_c(\mu, \nu,\epsilon)\coloneqq \left\lbrace f\in\FC_c \, \mid\, \mu(f)+\nu(f^c)\geq \OT_c(\mu,\nu)-\epsilon\right\rbrace.
\end{align*}

\textit{Step 2. Simplifying the derivative.} It remains to show that 
	\begin{equation} \label{eq:OmitEpsilonToZero} 
    	\lim_{\epsilon \searrow 0} \sup_{ f \in S_c(\mu, \nu, \epsilon)} \left( \Delta_{\mu}(f) +  \Delta_{\nu}(f\powerC) \right) =  \sup_{ f \in S_c(\mu, \nu)} ( \Delta_{\mu}(f) +  \Delta_{\nu}(f\powerC) ).
	\end{equation}
	The left hand side in \eqref{eq:OmitEpsilonToZero} is greater or equal to the right hand side since $S_c(\mu, \nu)\subseteq S_c(\mu, \nu, \epsilon)$ for all $\epsilon >0$. For the converse, take a positive decreasing sequence $(\epsilon_n)_{n \in \N}$ with $\epsilon_n \searrow 0$ and a sequence  $(f_n)_{n \in \N} \subseteq \FC_{c}$ with $f_n \in S_c(\mu, \nu, \epsilon_n)$ for which 
	\begin{equation*}
        \sup_{ f \in S_c(\mu, \nu, \epsilon_n)} ( \Delta_{\mu}(f) +  \Delta_{\nu}(f\powerC) ) - \epsilon_n \leq  \Delta_{\mu}(f_n) +  \Delta_{\nu}(f_n\powerC) . 
    \end{equation*}
    As we prove below there exists a subsequence  $(f_{n_k})_{k\in \N}$ such that $f_{n_k}$ and $f_{n_k}\powerC$ converge pointwise for $k \rightarrow \infty$ on $\supp(\mu)$ and $\supp(\nu)$ to functions  $h+a$ and $h\powerC -a$, respectively, for $h\in S_c(\mu, \nu)$ and $a \in \R$. 
    Once we show that \begin{equation}\label{eq:ConvergenceDelta_fn}\lim_{k \rightarrow \infty} ( \Delta_{\mu}(f_{n_k}) +  \Delta_{\nu}(f_{n_k}\powerC) ) = ( \Delta_{\mu}(h) +  \Delta_{\nu}(h\powerC) ),\end{equation}
    equality \eqref{eq:OmitEpsilonToZero} follows from
    \begin{align*}
        	\lim_{\epsilon \searrow 0} \sup_{ f \in S_c(\mu, \nu, \epsilon)} ( \Delta_{\mu}(f) +  \Delta_{\nu}(f\powerC)) &\leq \lim_{k \rightarrow \infty} \Delta_{\mu}(f_{n_k}) +  \Delta_{\nu}(f_{n_k}\powerC) \\& = ( \Delta_{\mu}(h) +  \Delta_{\nu}(h\powerC) )\leq \sup_{ f \in S_c(\mu, \nu)} ( \Delta_{\mu}(f) +  \Delta_{\nu}(f\powerC) ).
    \end{align*}
    To verify \eqref{eq:ConvergenceDelta_fn}, let $\delta>0$ and select $M \in \N$ such that  $\norm{\Delta_{\mu} - \Delta_{\mu,M}}_{\FC_c}<\delta/4$. Pointwise convergence of $f_{n_k}$ on $\supp(\mu)$ to $h+a$ combined with the uniform bound on $\FC_{c}$ asserts by dominated convergence for $\mu_M, \mu$ by $\supp(\mu_M)\subseteq \supp(\mu)$ existence of $K \in \N$ with $$\left|\mu_M (f_{n_k} - (h+a))\right| + \left|\mu\big(f_{n_k} - (h+a)\big)\right| <  \delta t_M/2 \quad \forall k \geq K.$$ Hence, for all $k \geq K$ it follows that 
    \begin{align*}
	    |\Delta_\mu(f_{n_k}) - \Delta_\mu(h)| &\leq 2\norm{ \Delta_\mu- \Delta_{\mu,M}}_{\FC_c} + |\Delta_{\mu, M}(f_{n_k}) - \Delta_{\mu, M}(h)|\\
	    &= 2\norm{ \Delta_{\mu}- \Delta_{\mu, M}}_{\FC_c} + t_M^{-1}\left|(\mu_M-\mu)\big(f_{n_k} - h\big)\right|\\
	    &= 2\norm{ \Delta_{\mu}- \Delta_{\mu, M}}_{\FC_c} + t_M^{-1}\left|(\mu_M-\mu) \big(f_{n_k} - (h-a)\big)\right| < \delta,
	\end{align*}
	where we use in the second equality the definition of $\Delta_{\mu,M}$ and in the last equality that $(\mu - \mu_M)(a) = 0$. Repeating the argument for $|\Delta_\nu(f_{n_k}\powerC) - \Delta_\nu(h\powerC)|$ yields \eqref{eq:ConvergenceDelta_fn}.\\

\textit{Step 3. Existence of converging subsequences.} 
We prove existence of a subsequence of $(f_n, f_n\powerC)$ that converges uniformly on compact sets to a pair of continuous functions $(f,g)$. Uniform convergence on compact sets of continuous functions on $\XC$ is induced by the compact-open topology which is metrizable since $\XC$ is a locally compact Polish space \cite[page 68]{mccoy1988topological}.
% To verify existence of a converging subsequence of $f_n$,
We show that $\FC_c$ is relatively compact in the compact-open topology by means of a general version of the Arzel\`a-Ascoli theorem.\\[0.3cm]
\begin{minipage}{.03\linewidth}\hspace{0.1cm}\end{minipage}%
\begin{minipage}{.94\linewidth} \begin{fact}[{\citealt[Theorem 3.2.6]{mccoy1988topological}}]
\label{thm:GeneralArzelaAscoli}
If $\XC$ is locally compact, then a set of continuous real-valued functions on $\XC$ is compact in the compact-open topology if and only if it is closed, pointwise bounded and equicontinuous.
\end{fact}\end{minipage}\\[0.2cm]
Since $\{c(\cdot, y) \mid y \in \YC\}$ is equicontinuous, so is $\FC_{c}$ and its closure $\closure(\FC_{c})$ in the compact-open topology. Moreover, since $\FC_{c}$ is uniformly bounded by $\norm{c}_\infty$, \Cref{thm:GeneralArzelaAscoli} asserts the existence of a subsequence of $f_n$ that converges uniformly on compact sets to a continuous function $f\in \closure(\FC_{c})$. By restricting to a subsequence, we assume that $f_n$ uniformly converges on compact sets to $f$. Under Assumption \ref{ass:EquicontinuityCompact}, this asserts that $f_n$ uniformly converges on $\XC$ to $f$ and thus $f_n \powerC$ uniformly converges on $\YC$ to $g\coloneqq f\powerC$. Under Assumption \ref{ass:EquicontinuityLocallyCompact}, we instead repeat the above argument and assume by local compactness of $\YC$ and equicontinuity of $\{c(x, \cdot) \mid x \in \XC\}$ that $f_n\powerC$ uniformly converges on compact sets of $\YC$ to some continuous function $g \in \closure(\FC_{c}\powerC)$.\\

\textit{Step 4. Structural properties of limits with general OT theory.} 
It remains to show that there exists a $c$-concave function $h\in S_c(\mu, \nu)$ and some $a \in \R$ such that $f=h+a$ on $\supp(\mu)$ and $g = h\powerC-a$ on $\supp(\nu)$. For this purpose, note that uniform convergence on compact sets implies pointwise convergence. 
    Since $f_n, f, f_n\powerC, g$ are all absolutely bounded by $\norm{c}_\infty$, we find by dominated convergence that 
    \begin{equation*}
        \mu(f) + \nu(g) = \lim_{n\to\infty} \mu(f_n) + \nu(f_n^c) \ge \lim_{n\to\infty} \OT_c(\mu, \nu) - \epsilon_n = \OT_c(\mu, \nu).
    \end{equation*}
    Moreover, for $(x,y) \in \XC\times \YC$ we find that \begin{align}\label{eq:FandG_are_feasible}
        f(x) + g(y) = \lim_{n \rightarrow \infty} f_n (x) + \lim_{n \rightarrow \infty} f_n\powerC(y) = \lim_{n \rightarrow \infty} f_n(x) + f_n\powerC(y) \leq c(x,y),
\end{align}
which asserts by the dual formulation for the OT cost \cite[Theorem 5.9]{villani2008optimal} that $\mu(f) + \nu(g) =\OT_c(\mu, \nu).$ Upon defining $\tilde h \coloneqq g\powerC$, we therefore obtain from \eqref{eq:FandG_are_feasible} the inequalities $f\leq \tilde h$ on $\XC$ and $g \leq \tilde h\powerC$ on $\YC$. Hence, it holds that $$ \OT_c(\mu, \nu)=  \mu(f) + \nu(g) \leq\mu( \tilde h) + \nu(g) \leq  \mu( \tilde h) + \nu( \tilde h\powerC)\leq\OT_c(\mu, \nu),$$ where the last inequality follows from $ \tilde h(x) +  \tilde h\powerC(y) \leq c(x,y)$ for all $(x,y)\in \XC\times \YC$. 
We thus conclude that $f = \tilde h$ holds $\mu$-almost surely and $g = \tilde h\powerC$ holds $\nu$-almost surely. 

Under Assumption \ref{ass:EquicontinuityLocallyCompact} it follows from step three that both $f$ and $\tilde h$ as well as $g$ and $\tilde h^c$ are continuous on $\XC$ and $\YC$, respectively. Thus, it holds (deterministically) that $f = \tilde h$ on $\supp(\mu)$ and $g = \tilde h\powerC$ on $\supp(\nu)$. 
Likewise, under Assumption \ref{ass:EquicontinuityCompact} it follows that $f$ and $\tilde h$ are continuous on $\XC$ which yields $f = \tilde h$ on $\supp(\mu)$.
Further, from step three we know  under \ref{ass:EquicontinuityCompact} that $g = f^c$ on $\YC$, i.e., $g$ is $c$-concave and hence  $g = g^{cc} = \tilde h^{c}$  on $\YC$ by \citet[Proposition 1.34]{santambrogio2015optimal}.
 Finally, we note by \citet[Remark 1.13]{vil03} that any $c$-concave Kantorovich potential $\tilde h$ can be suitably shifted by a constant $a\in \R$ such that $0 \leq \tilde h(x) - a \leq \norm{c}_{\infty}$ for all $x \in \XC$ and $-\norm{c}_{\infty}\leq (\tilde h -a)\powerC(y)= \tilde h\powerC(y)+a\leq 0$ for all $y\in  \YC$. In particular, the function $h \coloneqq \tilde h - a$ lies in $S_c(\mu, \nu)$ and fulfills the asserted properties.
\end{proof}

\begin{remark}\label{rem:ConstantMightBeNecessary}
In the proof of \Cref{thm:OTHadamardDiff}, we cannot guarantee that the $c$-concave function $\tilde h = g \powerC$ lies in $S_c(\mu, \nu)\subseteq \FC_c$ and therefore have to shift it by a suitable constant $a$. This is due to the inequalities $-\norm{c}_\infty \leq g \leq 0$ being possibly violated.
\end{remark}

\section{Normal Limits under Unique Kantorovich Potentials}\label{sec:NormalLimitLaws}

The set of Kantorovich potentials $S_c(\mu, \nu)$ in \eqref{eq:KantorovichPotentials} and its $c$-conjugates $S_c^c(\mu, \nu)$ play a defining role for the limiting random variables in \Cref{thm:GeneralCLT}. A particular setting of interest arises if Kantorovich potentials are uniquely determined. However, since any $f\in S_c(\mu, \nu)$ can be shifted arbitrarily and $f+a$ for $a\in\R$ is still a $c$-concave function that solves \eqref{eq:OTdual}, the set $S_c(\mu, \nu)$ can only be a singleton up to additive constants. Furthermore, it suffices if uniqueness holds almost surely.

\begin{definition}[Unique Kantorovich potentials]\label{def:UniquePotentials}
    The Kantorovich potentials $S_c(\mu, \nu)$ in \eqref{eq:KantorovichPotentials} are said to be \emph{unique} if $f_1-f_2$ for all $f_1,f_2\in S_c(\mu, \nu)$ is constant $\mu$-almost surely. 
\end{definition}

Notably, one can show that uniqueness of Kantorovich potentials $S_c(\mu, \nu)$ with respect to $\mu$ is in fact equivalent to uniqueness of the $c$-conjugate Kantorovich potentials $S_c^c(\mu, \nu)$ with respect to $\nu$ \cite[Lemma 5]{staudt2021uniqueness}.
Under this form of uniqueness, the limit random variables for the empirical OT cost simplify and follow a centered normal distribution.

\begin{theorem}[Normal limits]\label{thm:NormalLimits}
    Consider Polish spaces $\XC$ and $\YC$ and a cost function $c\colon \XC\times \YC\rightarrow\R_+$ satisfying Assumption \ref{ass:CostsContinuousAndBounded} combined with \ref{ass:EquicontinuityLocallyCompact} or \ref{ass:EquicontinuityCompact}.
     Assume that the Kantorovich potentials $S_c(\mu, \nu)$ for $\mu\in \PC(\XC)$ and $\nu \in\PC(\YC)$ are unique and let $f\in S_c(\mu, \nu)$. 
   Then, for empirical measures $\hat \mu_n, \hat\nu_m$ the following CLT is valid.
    \begin{subequations}\label{eq:GeneralCLT_normal}
    \begin{enumerate}
        \item[$(i)$](One-sample from $\mu$) Suppose that $\FC_c$ is $\mu$-Donsker. Then, for $n \rightarrow \infty$, \begin{equation}\label{eq:GeneralCLT_OneSampleMu_normal}
            \sqrt{n}\left(\OT_c(\hat \mu_n, \nu) - \OT_c(\mu, \nu)\right) \konvD \GG_\mu(f) \sim  \NC(0, \Var_{X\sim\mu}[f(X)]).
        \end{equation}
        \item[$(ii)$](One-sample from $\nu$) Suppose that $\FC_c\powerC$ is $\nu$-Donsker. Then, for $n \rightarrow \infty$, \begin{equation}\label{eq:GeneralCLT_OneSampleNu_normal}
          \sqrt{m}\left(\OT_c(\mu, \hat\nu_m) - \OT_c(\mu, \nu)\right) \konvD  \GG_\nu(f\powerC) \sim \NC(0, \Var_{Y\sim\nu}[f\powerC(Y)]).
        \end{equation}
        \item[$(iii)$](Two-sample) Suppose that $\FC_c$ is $\mu$-Donsker and that $\FC_c\powerC$ is $\nu$-Donsker. Then, for $n,m \rightarrow \infty$ with $m/(n+m)\rightarrow \delta\in (0,1)$,  
            \begin{equation}
  \label{eq:GeneralCLT_TwoSample_normal}\begin{aligned}
                \sqrt{\frac{nm}{n+m}}\left(\OT_c(\hat \mu_n, \hat\nu_m) - \OT_c(\mu, \nu)\right) \konvD  \sqrt{\delta}\, \GG_\mu(f) + \sqrt{1-\delta}\,\GG_\nu(f\powerC)& \\ \qquad \quad\sim \NC\left(0, \delta\Var_{X\sim\mu}[f(X)] + (1-\delta)\Var_{Y\sim\nu}[f\powerC(Y)]\right)&.
        \end{aligned}
\end{equation}
    \end{enumerate}
    \end{subequations}
\end{theorem}

\begin{proof}
Due to  \citet[Lemma~5]{staudt2021uniqueness} and the continuity of Kantorovich potentials in this setting (\Cref{lem:ContinuityOfKantPotentials}), 
we find that the set of Kantorovich potentials $S_c(\mu, \nu)$ and its $c$-conjugate counterparts $S_c\powerC(\mu, \nu)$ are deterministically unique (up to constant shifts) on the support of $\mu$ and $\nu$, respectively. Since any $\Delta_\mu\in \TC_{\mu}\big(\PC(\tilde \XC)\big)$ fulfills by step two of the proof for \Cref{thm:OTHadamardDiff} the (deterministic) equality  $\Delta_{\mu}(f) = \Delta_{\mu}(f+a)$ for $f\in \FC_{c}$ and $a \in\R$ with $f+a \in \FC_{c}$, and likewise for $\Delta_{\nu}\in \TC_{\nu}(\PC(\tilde \YC))$, uniqueness of Kantorovich potentials implies linearity of the directional Hadamard derivative. The functional delta method \citep{Roemisch04} thus implies the weak limit for the empirical OT cost to be centered normal with variance as stated in \eqref{eq:GeneralCLT_normal}.
\end{proof}

A useful statistical application of \Cref{thm:NormalLimits} arises when the limit variance $\Var_{X\sim\mu}[f(X)]$ in \Cref{thm:NormalLimits} can consistently be estimated from i.i.d.\ data. Indeed, this holds under Assumption \ref{ass:CostsContinuousAndBounded} combined with \ref{ass:EquicontinuityLocallyCompact} or \ref{ass:EquicontinuityCompact}, leading to the following pivotal limit law. 

\begin{corollary}[Pivotal limit law]\label{cor:invPrincipleOT}
If the unique Kantorovich potential $f\in S_{c}(\mu,\nu)$ in \Cref{thm:NormalLimits} is not constant $\mu$-almost surely, then, for $n\to\infty$ and any $f_n\in S_{c}(\hat{\mu}_n,\nu)$,
\begin{equation}\label{eq:invPrincipleOT}
     \sqrt{n}\, \frac{\OT_c(\hat \mu_n,\nu) - \OT_c(\mu,\nu)}{\sqrt{\Var_{X\sim\hat{\mu}_n}[f_n(X)]}}  \konvD \frac{\GG_\mu(f)}{\sqrt{\Var_{X\sim\mu}[f(X)]}} \sim \normal\left(0, 1\right).
\end{equation}
Analogous statements hold for the weak limits in \eqref{eq:GeneralCLT_OneSampleNu_normal} if $f\powerC\in S_c\powerC(\mu, \nu)$ is not constant \qquad $\nu$-almost surely and in \eqref{eq:GeneralCLT_TwoSample_normal} if $f$ or $f\powerC$ is not constant $\mu$- respectively $\nu$-almost surely. 
\end{corollary} 

\begin{proof}[Proof of \Cref{cor:invPrincipleOT}]
We only show the claim for \eqref{eq:invPrincipleOT},  the corresponding pivotal limits for \eqref{eq:GeneralCLT_OneSampleNu_normal} and \eqref{eq:GeneralCLT_TwoSample_normal} follow analogously. In view of \Cref{thm:NormalLimits} and Slutzky's lemma it suffices to show for $f_n \in S_c(\hat\mu_n, \nu)$ that $\Var_{X\sim\hat{\mu}_{n}}[f_{n}(X)]$ converges almost surely for $n \rightarrow \infty$ to $\Var_{X\sim\mu}[f(X)]>0$ for  $f \in S_c(\mu, \nu)$. Note that $S_c(\hat\mu_n, \nu) \subseteq S_c(\mu, \nu, 2 \norm{\hat\mu_n - \mu}_{\FC_c})$ where $\norm{\hat\mu_n - \mu}_{\FC_c}$ tends to zero almost surely since $\FC_c$ is $\mu$-Donsker. Hence, by steps three and four of the proof for \Cref{thm:OTHadamardDiff}, it follows that there exists a subsequence $(f_{n_k})_{k \in \N}$ such that $(f_{n_k}, f_{n_k}\powerC)$ converges pointwise on $\supp(\mu)\times \supp(\nu)$ for $k \rightarrow \infty$ to $(h + a, h\powerC -a)$ for some $h \in S_c(\mu, \nu)$ and some $a \in \R$. Since $\FC_c$ is uniformly bounded by $\norm{c}_\infty$, and since $\FC_c$ as well as the element-wise squared function class $\FC_c^2 =\{f^2 \colon f \in \FC_c\}$ are both $\mu$-Donsker \cite[Theorem 2.10.6]{van1996weak}, we conclude that $\Var_{X\sim\hat{\mu}_{n_k}}[f_{n_k}(X)] \rightarrow \Var_{X\sim\mu}[h(X)]$ almost surely for $k \rightarrow \infty$. Finally, by almost sure uniqueness of Kantorovich potentials it holds that $\Var_{X\sim\mu}[h(X)] = \Var_{X\sim\mu}[f(X)]$.
\end{proof}

\paragraph*{Uniqueness of Kantorovich potentials}
The recent work by \cite{staudt2021uniqueness} provides a detailed account on uniqueness guarantees of Kantorovich potentials. The gist of their article is that uniqueness in settings with differentiable costs will commonly hold, and that non-differentiabilities, erratic mass placement, or specific degeneracies in disconnected spaces have to be assumed for uniqueness to fail.

In the connected setting, one relies on the fact that the gradient of a Kantorovich potential at some point $x$ is determined by the gradient of $c(\cdot, y)$ at $x$ for any $(x, y) \in \supp(\pi)$, where $\pi$ is an OT plan, i.e., an optimizer of \eqref{eq:OT}. The connectedness of the support of $\mu$ combined with the uniqueness of the gradient then imply uniqueness of the potential. 
This approach was employed by \cite{staudt2021uniqueness} in case of probability measures with possibly unbounded support that assign negligible mass to the boundary of their support and differentiable cost functions that grow sufficiently rapidly. Notably, this encompasses the strictly convex costs considered by \cite{gangbo1996geometry}, for which uniqueness guarantees were also derived by \citet[Theorem 2.4]{delBarrio2021GeneralCosts} and \citet[Theorem B.2]{bernton2021entropic}.
In discrete settings, the theory of linear programming shows Kantorovich potentials to be unique if the measures $\mu$ and $\nu$ are non-degenerate, which (loosely speaking) means that the OT problem cannot be divided into proper sub-problems. \cite{staudt2021uniqueness} show that similar results hold for probability measures with disconnected support on Polish spaces. 

\section{Degenerate Limits under Trivial Kantorovich Potentials}\label{sec:DegenerateLimitLaws}

Another important special case for our CLTs emerges if Kantorovich potentials are not only unique but also constant. The asymptotic variance in \Cref{thm:NormalLimits} is then equal to zero and the limit law degenerates. As for uniqueness, it suffices if Kantorovich potentials are constant in an almost sure sense.

\begin{definition}[Trivial Kantorovich potentials]
  A Kantorovich potential $f\in S_c(\mu, \nu)$ is called \emph{trivial} if it is constant $\mu$-almost surely. The set $S_c(\mu, \nu)$ is said to be trivial if all of its elements are trivial.
 \end{definition}

The same definition applies for conjugated potentials $f^c$ and the set $S_c\powerC(\mu, \nu)$ with respect to $\nu$.
In contrast to uniqueness of Kantorovich potentials, triviality of $S_c(\mu, \nu)$ does not imply the triviality of $S_c^c(\mu, \nu)$, since a $\mu$-almost surely constant $f$ can (and often will) have a $c$-conjugate $f^c$ that is not $\nu$-almost surely constant (see also \Cref{rem:DoubleTrivialPotential} below).
Simple examples with trivial Kantorovich potentials occur if one of the measures $\mu$ or $\nu$ is a Dirac measure or if the cost function is itself constant. More general and interesting settings will be discussed below.

\begin{theorem}[Degenerate limits]\label{thm:DegenerateOTLimits}
Consider Polish spaces $\XC$ and $\YC$ and a cost function $c\colon \XC\times \YC\rightarrow\R_+$ satisfying Assumption \ref{ass:CostsContinuousAndBounded} combined with \ref{ass:EquicontinuityLocallyCompact} or \ref{ass:EquicontinuityCompact}. Let $\mu \in \PC(\XC)$ and $\nu \in \PC(\YC)$ be probability measures. 
 \begin{enumerate}
        \item[$(i)$] (One-sample from $\mu$) Suppose that $\FC_c$ is $\mu$-Donsker. Then, the limit law \eqref{eq:GeneralCLT_OneSampleMu} degenerates to a Dirac measure if and only if  $S_c(\mu,\nu)$ is trivial.
        \item[$(ii)$](One-sample from $\nu$) Suppose that $\FC_c\powerC$ is $\nu$-Donsker. Then, the limit law \eqref{eq:GeneralCLT_OneSampleNu} degenerates to a Dirac measure  if and only if $S_c\powerC(\mu, \nu)$ is trivial.
        \item[$(iii)$](Two-sample) Suppose that $\FC_c$ and $\FC_c\powerC$ are $\mu$- and $\nu$-Donsker, respectively. Then, the limit law \eqref{eq:GeneralCLT_TwoSample} degenerates to a Dirac measure if and only if  $S_c(\mu, \nu)$  and $S_c\powerC(\mu, \nu)$ are both trivial.
    \end{enumerate}
\end{theorem}

\begin{proof}
We only state the proof for setting $(i)$,  the remaining cases follow analogously.  If~$S_c(\mu, \nu)$ is trivial, then the limit distribution degenerates by \Cref{thm:NormalLimits}$(i)$ to a Dirac measure at zero. In case $S_c(\mu, \nu)$ is not trivial, there exists a Kantorovich potential $f\in S_c(\mu, \nu)$ with $\Var_{X\sim \mu}[f(X)]>0$. As the limit distribution for the one-sample case from $\mu$ stochastically dominates $\NC(0,\Var_{X\sim \mu}[f(X)])$ it does not degenerate to a Dirac measure.
\end{proof}

The existence of trivial Kantorovich potentials $f\in S_c(\mu, \nu)$ is intimately related to the existence of transport plans that act as projections onto the support of $\mu$. To highlight the underlying geometric interpretation, we introduce the following notion of projected measures.

\begin{definition}[Projected measures]\label{def:projectedmeasure}
Let $\mu\in\PC(\XC)$ and $\nu\in\PC(\YC)$ be probability measures. We say that $\mu$ is a \emph{$\nu$-projected measure} (with respect to $c$) and write $\mu\in P_c(\nu)$ if there exists a coupling $\pi\in\Pi(\mu, \nu)$ such that
\begin{equation}\label{eq:ProjectedMeasure}
      c(x, y) = \inf_{x'\in\supp(\mu)}c(x', y)
      \qquad\text{for all}~(x, y)\in\supp(\pi).
\end{equation}
\end{definition}

Analogous definitions apply for $\mu$-projected measures, which we denote by $\nu\in P_c(\mu)$.
It can easily be verified that any coupling $\pi$ satisfying \eqref{eq:ProjectedMeasure} solves the OT problem in \eqref{eq:OT} between $\mu$ and $\nu$. In fact, under continuous costs, it holds that (see \Cref{sec:Proofs} for a proof)
\begin{equation}\label{eq:ProjectedMeasureOT}
  \mu\in P_c(\nu)
  \qquad\Longleftrightarrow\qquad
  \OT_c(\mu, \nu)
  =
  \int_\YC \inf_{x\in\supp(\mu)} c(x, y)\dif\nu(y),
\end{equation}
which can equivalently be used to characterize $\nu$-projected measures.

\begin{theorem}[Existence of trivial Kantorovich potentials]\label{thm:ConstantPotentials}
Consider Polish spaces $\XC$ and $\YC$ and a cost function $c\colon \XC\times \YC\rightarrow\R_+$ satisfying Assumption \ref{ass:CostsContinuousAndBounded} combined with \ref{ass:EquicontinuityLocallyCompact} or \ref{ass:EquicontinuityCompact}, and let $\mu\in \PC(\XC)$ and $\nu\in\PC(\YC)$ be probability measures.
Then, trivial Kantorovich potentials $f\in S_c(\mu, \nu)$ exist if and only if $\mu\in P_c(\nu)$. 
\end{theorem}

\begin{remark}\label{rem:UniquenessGivesSharpness}
  In settings where we can assume unique Kantorovich potentials, which is often not a restrictive condition (see \Cref{sec:NormalLimitLaws}), \Cref{thm:ConstantPotentials} in conjunction with \Cref{thm:DegenerateOTLimits} states that the one-sample limit laws \eqref{eq:GeneralCLT_OneSampleMu} and \eqref{eq:GeneralCLT_OneSampleNu} degenerate if and only if $\mu\in P_c(\nu)$ or $\nu\in P_c(\mu)$, respectively. Furthermore the two-sample limit law \eqref{eq:GeneralCLT_TwoSample} degenerates if and only if $\mu\in P_c(\nu)$ and $\nu\in P_c(\mu)$ (see \Cref{rem:DoubleTrivialPotential} below).
\end{remark}

Intuitively, a $\nu$-projected measure is any measure that can be obtained from $\nu$ by projecting all points of $\supp(\nu)$ to some subset of $\XC$ according to the cost $c$. For example, $\mu\in P_c(\nu)$ always holds if $\mu = p_\#\nu$ for a measurable map $p\colon\supp(\nu) \to\XC$ that obeys
\begin{equation}\label{eq:CProjection}
  p(y) \in \argmin_{x\in\supp(\mu)} c(x, y)
\end{equation}
for each $y\in\supp(\nu)$. We denote maps $p$ that satisfy \eqref{eq:CProjection} for given $\mu$ and $\nu$ as \emph{$c$-projections}. One example of a $c$-projection is illustrated in \Cref{fig:ConstantPotentials}$(a)$.
Based on \Cref{thm:ConstantPotentials}, we next formulate several necessary and sufficient criteria for the existence of trivial Kantorovich potentials, all of which have an intuitive geometric interpretation. 

\begin{figure}
  \begin{center}\small
  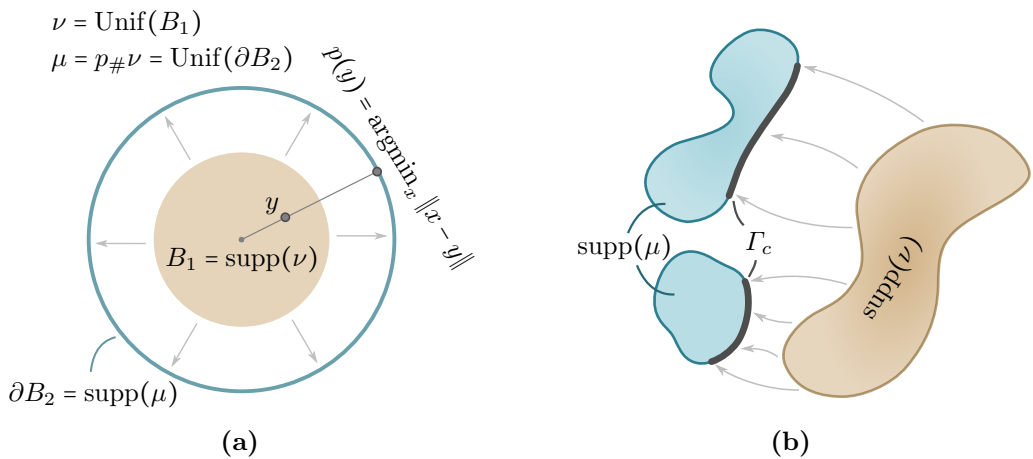
    \vspace{-2ex}
  \end{center}
  \caption{Trivial Kantorovich potentials. In \textbf{(a)}, $\mu$ is equal to the projection of $\nu$ onto the set $\tilde\XC = \partial B_2$, the boundary of an Euclidean ball. According to \Cref{cor:ConstantPotentials1} $(ii)$, there exists a Kantorovich potential $f$ that is constant on $\partial B_2$. Furthermore, due to \citet[Corollary~2]{staudt2021uniqueness}, the constant potential is also unique. In \textbf{(b)}, the projection $\Gamma_c= \Gamma_c(\mu, \nu)$ of all points $y\in\supp(\nu)$ onto the support of $\mu$ under Euclidean costs is marked by black segments that are part of the boundary of $\supp(\mu)$. Since $\supp(\mu)$ is not contained in $\Gamma_c$, \Cref{cor:ConstantPotentials2} $(iii)$ states that there cannot be a Kantorovich potential $f$ that is constant on $\supp(\mu)$.}
  \label{fig:ConstantPotentials}
\end{figure}

\begin{corollary}\label{cor:ConstantPotentials1}
 Under the assumptions of \Cref{thm:ConstantPotentials}, the set $S_c(\mu, \nu)$ contains trivial Kantorovich potentials if one of the following conditions is satisfied.
  \begin{enumerate}
    \item[$(i)$] $\mu = \nu$ and $c(x, x) = 0$ holds for all $x\in\XC$,
    \item[$(ii)$] $\mu = p_{\#}\nu$ for a $c$-projection $p\,\colon\supp(\nu)\to \tilde\XC$ onto a closed set $\tilde\XC\subset \XC$.
  \end{enumerate}
 \end{corollary}

\begin{corollary}
\label{cor:ConstantPotentials2}
 Under the assumptions of \Cref{thm:ConstantPotentials}, the set $S_c(\mu, \nu)$ does not contain trivial Kantorovich potentials if one of the following conditions is satisfied.
  \begin{enumerate}
    \item[$(i)$] $\mu\neq\nu$ while $\supp(\nu) \subset \supp(\mu)$ with $c(x, x) = 0$ and $c(x, y) > 0$ for all $x\neq y$,
    \item[$(ii)$] $\mathrm{int}(\mathrm{supp}(\mu))\nsubseteq \mathrm{supp}(\nu)$ for subsets $\XC,\YC\subset V$ of a normed linear space $(V, \|\cdot\|)$ with cost $c(x, y) = h(\|x - y\|)$ for strictly increasing $h$,
    \item[$(iii)$] $\supp(\mu)$ is not equal to the topological closure $\Gamma_c(\mu, \nu) \subset\XC$ of the set
      \begin{equation*}
        \bigcup_{y\in\supp(\nu)} \argmin_{x\in\supp(\mu)} c(x, y).
      \end{equation*}
  \end{enumerate}
In each of these settings, the limit law \eqref{eq:GeneralCLT_OneSampleMu} in \Cref{thm:GeneralCLT} is non-degenerate.
\end{corollary}

Exchanging the roles of $\mu$ and $\nu$, \Cref{thm:ConstantPotentials} as well as Corollaries \ref{cor:ConstantPotentials1} and \ref{cor:ConstantPotentials2} can equally be applied to $f^c$. Examples illustrating \Cref{cor:ConstantPotentials1} $(ii)$ and \Cref{cor:ConstantPotentials2} $(iii)$ in Euclidean settings are provided in \Cref{fig:ConstantPotentials}. In particular, \Cref{fig:ConstantPotentials}$(a)$ highlights that there is a crucial difference between demanding constant Kantorovich potentials on all of $\XC$ and only demanding them to be constant almost surely. Indeed, if $f\in S_c(\mu, \nu)$ was constant on the whole space $\XC$, then $f^c\in S_c^c(\mu, \nu)$ would be constant as well, which is by \Cref{thm:ConstantPotentials} not the case in \Cref{fig:ConstantPotentials}$(a)$.
This setting also serves as an example where $\mu\neq\nu$ and where every Kantorovich potential $f\in S_c(\mu, \nu)$ is trivial while $f^c$ is not.

\begin{remark}[Bi-triviality]\label{rem:DoubleTrivialPotential}
  OT problems where both $f\in S_c(\mu, \nu)$ and its conjugate $f^c$ are trivial have a special underlying geometry. If $f$ is constant on $\supp(\mu)$ and $f^c$ is constant on $\supp(\nu)$, then the optimality condition $f(x) + f^c(y) = c(x, y)$ for points $(x, y)$ on the support of any optimal transport plan $\pi$ implies that $c$ is constant on $\supp(\pi)$, i.e., all points are transported with the same cost. For an example, consider radially symmetric costs and uniform distributions $\mu$ and $\nu$ on centered Euclidean spheres of positive radius. For different radii, the measures are distinct and both Kantorovich potentials are trivial.
\end{remark}

\section{Examples}\label{sec:LLspecificinstance}

Elaborating the main result in \Cref{thm:GeneralCLT}, we focus in this section on concrete settings. The general setup requires Assumption \ref{ass:CostsContinuousAndBounded} combined with \ref{ass:EquicontinuityLocallyCompact} or \ref{ass:EquicontinuityCompact} to hold. It remains to verify the Donsker properties for the function classes $\FC_c$ and $\FC_c^c$ which then leads to CLTs for the empirical OT cost. 

\subsection{Countable Discrete Spaces}\label{subsec:countable}

Our general theory leads to CLTs for the OT cost on countable discrete spaces $\XC$ and $\YC$ equipped with discrete topology. More precisely, for a non-negative cost function $c\colon \XC\times \YC \to \mathbb{R}_+$ bounded by some constant $\Vert c \Vert_\infty<\infty$, Assumptions \ref{ass:CostsContinuousAndBounded} and~\ref{ass:EquicontinuityLocallyCompact} hold trivially. To prove that the function class $\FC_c$ is $\mu$-Donsker, we define $\FC_c\mathbbm{1}_{x}$ as the restriction of $\FC_c$ to a fixed element $x\in\XC$. Notably, each element in the latter function class is bounded by $\Vert c \Vert_\infty$ and we obtain
\begin{equation*}
    \mathbb{E}\left[ \Vert \sqrt{n}\left(\hat{\mu}_n -\mu \right) \Vert_{\FC_c\mathbbm{1}_{x}} \right] \lesssim_{\Vert c \Vert_\infty} \sqrt{\mu(x)}.
\end{equation*}
According to \cite[Theorem 2.10.24]{van1996weak}, we deduce that the function class $\FC_c$ is $\mu$-Donsker if 
\begin{equation}\label{eq:BorisovDurst}
    \sum_{x\in \XC} \sqrt{\mu(x)}<\infty
\end{equation}
which is the celebrated \emph{Borisov-Dudley-Durst} condition \citep{dudley1999uniform}. Similarly, the function class $\FC_c^c$ is $\nu$-Donsker if $\sum_{y\in \YC} \sqrt{\nu(y)}<\infty$.

\begin{corollary}[Countable discrete spaces, \citealt{tameling18}]\label{cor:countableCLTs}
Let $\XC$ and $\YC$ be countable discrete spaces and $c\colon \XC\times \YC \to \mathbb{R}_+$ a bounded cost function. Consider probability measures $\mu\in \PC(\XC)$ and $\nu \in \PC(\YC)$. If $\mu$ fulfills the Borisov-Dudley-Durst condition \eqref{eq:BorisovDurst}, then the CLT in \eqref{eq:GeneralCLT_OneSampleMu} is valid. If $\nu$ fulfills \eqref{eq:BorisovDurst}, then \eqref{eq:GeneralCLT_OneSampleNu} holds. In case both $\mu$ and $\nu$ fulfill \eqref{eq:BorisovDurst}, then \eqref{eq:GeneralCLT_TwoSample} holds.
\end{corollary}

\subsection{The One-dimensional Euclidean Space}\label{subsec:d123}

From \Cref{thm:GeneralCLT} we immediately derive CLTs for the empirical OT cost between probability measures supported on the the real line $\RR$ and sufficiently regular cost function. The proofs require the notion of covering and metric entropy for a real-valued function class $\FC$ defined on $\XC$ as introduced in the notation. Based on metric entropy bounds, empirical process theory provides tools to assess if a given function class is $\mu$-Donsker or even universal Donsker \cite[Section 2.5]{van1996weak}. We first focus on the real line.

\begin{theorem}[$d=1$]\label{cor:CLT_RealLine}
Consider the Euclidean space $\R$ with cost $c\colon \R \times \R \rightarrow \R_+$ assumed to be bounded and $(\alpha,L)$-H\"older for $\alpha\in(1/2, 1]$ and $L\geq 0$, i.e., 
\begin{equation}\label{eq:HoelderCosts}
    \left|c(x,y) - c(x',y')\right| \leq L\left(\left|x-x'\right|^\alpha +\left|y-y'\right|^\alpha\right) \quad \forall x,x', y,y' \in \R. 
\end{equation}
If the probability measure $\mu\in\PC(\R)$ fulfills 
\begin{equation}\label{eq:BorisovDurstRealLine}
\sum_{k \in \mathbb{Z}} \sqrt{\mu\left([k,k+1)\right)} < \infty,
\end{equation}
then the CLT in \eqref{eq:GeneralCLT_OneSampleMu} is valid. If $\nu\in\PC(\R)$ fulfills \eqref{eq:BorisovDurstRealLine}, then \eqref{eq:GeneralCLT_OneSampleNu} holds. In case both $\mu$ and $\nu$ fulfill \eqref{eq:BorisovDurstRealLine}, then \eqref{eq:GeneralCLT_TwoSample} holds.
\end{theorem}

\begin{proof}
By the assumptions imposed on the cost and since the setting focuses on the real line $\R$, the Assumptions \ref{ass:CostsContinuousAndBounded} and \ref{ass:EquicontinuityLocallyCompact} hold. Based on \Cref{thm:GeneralCLT}, it remains to prove the Donsker properties of the function classes $\FC_c$ and $\FC_c^c$ in this case. By \Cref{lem:OTpotentials} (i), the function classes $\FC_c$ and $\FC_c\powerC$ are both contained in the class of uniformly bounded $(\alpha,L)$-H\"older functions on $\R$. According to \citet[Example 2.10.25]{van1996weak} with $M_k=L$ for all $k\in\mathbb{Z}$ this class is $\mu$-Donsker if \eqref{eq:BorisovDurstRealLine} is fulfilled.
\end{proof}

\begin{example}[Costs $c(x,y) = h(x-y)$, $d =1$]\label{ex:ApplicationsDim1}
To demonstrate a few consequences of \Cref{cor:CLT_RealLine}, suppose the cost is equal to $c(x,y) = h(x-y)$ on $\RR$ for a suitable function $h\colon \RR\rightarrow \RR_+$. For instance, under a bounded $(\alpha,L)$-H\"older function $h\colon \RR\rightarrow \RR_+$ with $\alpha>1/2$ and $L>0$ condition \eqref{eq:HoelderCosts} is satisfied. Hence for $\mu, \nu \in \PC(\RR)$ that fulfill condition \eqref{eq:BorisovDurstRealLine} \Cref{cor:CLT_RealLine} yields CLTs for the empirical OT cost.
	This setting encompasses, e.g., thresholded costs $c_{p,T}(x,y) = \min(|x-y|^p, T)$ for $p>\alpha$ and $T>0$. Notably, for $p = T = 1$, the function class $\FC_c$ coincides by Kantorovich-Rubinstein duality with $\BL{\RR}$, the class of $1$-Lipschitz functions uniformly bounded by one, for which condition \eqref{eq:BorisovDurstRealLine} is necessary and sufficient in order to be $\mu$-Donsker \cite[Theorem 1]{gine1986empirical}. For this case  it holds for $n\rightarrow \infty$ that 
\begin{equation*}
    \sqrt{n}\OT_{c_{1,1}}(\hat \mu_n, \mu)  \konvD \sup_{f \in \BL{\RR}} \mathbb{G}_\mu(f),
\end{equation*}
where the limit distribution only degenerates if $\mu$ is a Dirac measure. 

When the probability measures $\mu, \nu$ are compactly supported, the summability condition \eqref{eq:BorisovDurstRealLine} is trivially fulfilled and by restricting to the support of the probability measures it  suffices that $h$ is only \emph{locally} $\alpha$-H\"older for $\alpha\in (1/2,1]$. This provides for costs $c(x,y) = |x-y|^p$ with $p>1/2$ novel CLTs for the empirical OT cost. In particular, if the support of $\mu$ is disconnected, then \citet[Lemma 11]{staudt2021uniqueness} assert  existence of non-trivial potentials for $S_c(\mu, \mu)$ which implies the resulting limit distribution of $\sqrt{n}\OT_c(\hat \mu_n, \mu)$ to be non-degenerate (\Cref{thm:DegenerateOTLimits}). In contrast, if $\supp(\mu)$ is the closure of an open connected set for $p>1$ Kantorovich potentials  $S_c(\mu, \mu)$ are unique \cite[Corollary 2]{staudt2021uniqueness} and trivial and thus the respective limit distribution degenerates, indicating a faster convergence rate (see e.g. for $p=2$ the CLT by \cite{del2005asymptotics} which requires additional regularity assumptions on a density of $\mu$ and scaling by $n$). 
\end{example}

\begin{example}[Kantorovich-Rubinstein duality]
For Euclidean costs $c(x,y)=|x-y|$ and a compactly supported probability measure $\mu$ set $\XC = \YC = \supp(\mu)\subseteq \RR$. Then, the set $\FC_c = S_1(\mu, \mu) = \mathrm{Lip}_1(\XC)$ consists of $1$-Lipschitz functions on $\XC$ which are absolutely bounded by the diameter of $\XC$. In particular, we obtain for $n\rightarrow \infty$ that
\begin{equation*}
    \sqrt{n}\OT_{1}(\hat \mu_n, \mu)  \konvD \sup_{f \in \mathrm{Lip}(\XC)} \mathbb{G}_\mu(f).\notag
\end{equation*}
The approach presented here is based on the dual formulation of the OT cost. Using the primal perspective, which provides for $p=1$ an explicit formula for $\OT_1(\cdot,\cdot)$, \cite{del1999central} derived a CLT that also holds for $\mu$ with non-compact support (see also \citealt{mason2016weighted}). If $F_\mu$ denotes the cumulative distribution function of $\mu$ such that
\begin{equation}\label{eq:J1functional}
    \int_{-\infty}^{\infty} \sqrt{F_\mu(t)(1-F_\mu(t))}\dif t <  \infty,\notag
\end{equation}
then, for $\B_\mu(t)=\B(F_\mu(t))$ with $\B(t)$ a standard Brownian bridge and $n\rightarrow \infty$, it holds
\begin{equation}\label{eq:Wasserstein1WeakLimit}
    \sqrt{n}\OT_{1}(\hat \mu_n, \mu)   \konvD \int_{-\infty}^{\infty} \rvert \B_{\mu}(t)\lvert \dif t.
\end{equation}
Indeed, by suitably coupling the Gaussian processes $\GG_\mu$ and $\B_\mu$ and approximating the respective random element in terms of an unsigned measure, an application of Fubini's theorem shows for compactly supported $\mu$ with $\XC = \supp(\mu)$ that
\begin{equation}
    \int_{-\infty}^{\infty} \rvert \B_{\mu}(t)\lvert \dif t \overset{\mathcal{D}}{=}  \sup_{f \in \mathrm{Lip}(\XC)} \mathbb{G}_\mu(f).\notag
\end{equation}
\end{example}

\subsection{The Two- and Three-dimensional Euclidean Space} \Cref{thm:GeneralCLT} also characterizes the limit law for the empirical OT cost beyond the real line. For Euclidean spaces with dimension $d=2$ or $d = 3$ we obtain the following novel results.

\begin{theorem}[$d=2,3$]\label{cor:CLT_EuclSpace}
Consider the Euclidean space $\R^d$ for $d=2$ or $d = 3$ with cost $c\colon \R^d \times \R^d \rightarrow \R_+$ assumed to be bounded and $L$-Lipschitz\footnote{This refers to the cost function being $(1,L)$-H\"older, recall  \eqref{eq:HoelderCosts}.}. Further, suppose there exists some $\Lambda>0$ such that for all $k\in \mathbb{Z}^d$ there exist $x_k, y_k\in [k,k+1)$ such that
\begin{equation}\label{eq:locallySemiconcave}
  \begin{aligned}
     &c(\cdot, y) - \Lambda \norm{\cdot-x_k}_2^2 \text{ is concave on } [k,k+1)\text{ for all } y \in \R^d,\\
     &c(x, \cdot) - \Lambda \norm{\cdot-y_k}_2^2 \text{ is concave on } [k,k+1) \text{ for all } x \in \R^d.
 \end{aligned}
 \end{equation}
If the probability measure $\mu\in\PC(\R^d)$ fulfills 
 \begin{equation}\label{eq:BorisovDurstContinuousHigherD}
    \sum_{k \in \mathbb{Z}^d} \sqrt{\mu\left([k,k+1)\right)} < \infty,
\end{equation}
then the CLT in \eqref{eq:GeneralCLT_OneSampleMu} is valid. If $\nu\in\PC(\R^d)$ fulfills $\eqref{eq:BorisovDurstContinuousHigherD}$, then \eqref{eq:GeneralCLT_OneSampleNu} holds. In case both $\mu$ and $\nu$ fulfill \eqref{eq:BorisovDurstContinuousHigherD}, then \eqref{eq:GeneralCLT_TwoSample} holds.
\end{theorem}

\begin{proof}
By the assumptions imposed on the cost and since the setting focuses on the Euclidean space $\R^d$, the Assumptions \ref{ass:CostsContinuousAndBounded} and \ref{ass:EquicontinuityLocallyCompact} hold. Based on the main \Cref{thm:GeneralCLT}, it remains to prove the Donsker properties of the function classes $\FC_c$ and $\FC_c^c$ in this case. For a convex bounded set $\Omega\subseteq \R^d$ and constants $K,L>0$, let $\CC_{K, L}(\Omega)$ be the class of concave $L$-Lipschitz functions which are absolutely bounded by $K$. In order to prove that $\FC_c$ is $\mu$-Donsker if \eqref{eq:BorisovDurstContinuousHigherD} holds, we employ \citet[Theorem 2.10.24]{van1996weak}. To this end, consider a partition $\R^d = \bigcup_{k \in \mathbb{Z}^d}[k,k+1)$ and define the function class $\FC_{c,k}\coloneqq \FC_c\mathbbm{1}_{[k,k+1)}$. We first verify that each class $\FC_{c,k}$ is $\mu$-Donsker. First note for $k \in \mathbb{Z}^d$ and any element $f \in \FC_c$ that $f(\cdot) - \Lambda \norm{\cdot-x_k}_2^2$ is bounded, Lipschitz and concave on $[k,k+1)$ (see \Cref{lem:OTpotentials} $(ii)$). More precisely, for any $f \in\FC_c$ and since $x_k\in [k,k+1)$ it follows
\begin{equation*}
\left(f(\cdot)-\Lambda\norm{\cdot-x_k}_2^2\right)\!\Big|_{[k,k+1)}\in \CC_{\kappa,l}\left([k,k+1)\right)
\end{equation*} 
with $\kappa \coloneqq (\norm{c}_\infty + \Lambda d)$ and $l \coloneqq (L+2\Lambda d)$. According to \citet[Theorem 6]{bronshtein1976varepsilon}\footnote{The work by \cite{bronshtein1976varepsilon} in fact only provides metric entropy bounds for \emph{convex} bounded Lipschitz functions on a cube but of course they remain valid for concave functions.}, we conclude for $\epsilon>0$ sufficiently small
\begin{equation*}
\log\left(\NC(\epsilon,\FC_{c,k}, \norm{\cdot}_{\infty})\right)\leq \log\left(\NC(\epsilon,\CC_{\kappa,l}([k,k+1)), \norm{\cdot}_{\infty, [k,k+1)})\right) \lesssim_{\kappa,l,d}  \epsilon^{-d/2}.
\end{equation*}
Note that the function class $\FC_{c,k}$ has envelope function $F_{c,k}(\cdot)\coloneqq \norm{c}_\infty\mathbbm{1}_{[k,k+1)}(\cdot)$. Furthermore, an $\epsilon \norm{c}_\infty$-covering $\{f_1, \dots, f_N\}$ of $\FC_{c,k}$ with respect to $\norm{\cdot}_{\infty}$ defines for any finitely supported probability measure $\gamma \in \PC(\R^d)$ with $\norm{F_{c,k}}_{2,\gamma}>0$ an $\epsilon \norm{F_{c,k}}_{2,\gamma}$-covering for $\FC_{c,k}$ with respect to $\norm{\cdot}_{2,\gamma}$. 
Indeed, for $f \in \FC_{c,k}$ pick $f_i$ such that $\norm{f - f_i}_\infty< \epsilon\norm{ c }_\infty$ which yields
\begin{equation*}
     \norm{f - f_i}_{2,\gamma} \leq \sqrt{\int_{[k,k+1)} \epsilon^2\norm{ c }_\infty^2 \dif \gamma} = \epsilon \norm{c}_\infty \sqrt{\gamma([k,k+1))} = \epsilon \norm{F_{c,k}}_{2,\gamma}.
\end{equation*}
Hence, we conclude for any finitely supported $\gamma \in \PC(\R^d)$ with $\norm{F_{c,k}}_{2,\gamma}>0$ and sufficiently small $\epsilon$ that
\begin{align*}
\log \left(\NC( \epsilon \norm{F_{c,k}}_{2,\gamma}, \FC_{c,k}, \norm{\cdot}_{2, \gamma})\right) \leq \log \left(\NC(\epsilon \norm{c}_\infty, \FC_{c,k}, \norm{\cdot}_{\infty})\right) \lesssim_{\kappa, l,d} \epsilon^{-d/2}.
\end{align*}
After taking square roots the latter bound is integrable around zero for $d\leq  3$ which yields the $\mu$-Donsker property for $\FC_{c,k}$ \cite[Theorem 2.5.2]{van1996weak}. The $\mu$-Donsker property of the whole $\FC_c$ now follows if
\begin{equation}\label{eq:NecSummationConstraintForvdV}
\sup_{n \in \N}\sum_{k \in \mathbb{Z}^d}\mathbb{E}\left[\norm{\sqrt{n}(\hat\mu_n - \mu)}_{\FC_{c,k}}\right]< \infty.
\end{equation}
By standard chaining arguments each individual summand can be bounded by
\begin{align*}
\mathbb{E}\left[\norm{\sqrt{n}(\hat\mu_n - \mu)}_{\FC_{c,k}}\right] &\leq\int_{0}^{1}\sup_{\gamma}\sqrt{1+ \log\left(\NC(\epsilon \norm{F_{c,k}}_{2,\gamma}, \FC_{c,k}, \norm{\cdot}_{2, \gamma})\right)} \dif \epsilon \norm{F_{c,k}}_{2,\mu} \\
&\leq \int_{0}^{1}\sqrt{1+ \log\left(\NC(\epsilon \norm{c}_\infty, \FC_{c,k}, \norm{\cdot}_{\infty})\right)} \dif \epsilon \norm{F_{c,k}}_{2,\mu} \\
&\lesssim_{\kappa,l,d} \int_0^{1} \epsilon^{-d/4} \dif \epsilon  \norm{F_{c,k}}_{2,\mu} \lesssim_{d}  \norm{F_{c,k}}_{2,\mu} =  \norm{c}_\infty \sqrt{\mu([k,k+1))}. 
\end{align*}
Herein, the supremum runs over all finitely supported probability measures over $\mathbb{R}^d$ which leads to the claimed upper bound by our previous arguments. Summing over $k\in\mathbb{Z}^d$ and provided $\mu$ fulfills \eqref{eq:BorisovDurstContinuousHigherD} yields \eqref{eq:NecSummationConstraintForvdV}.
\end{proof}

The summability constraints \eqref{eq:BorisovDurstRealLine} and \eqref{eq:BorisovDurstContinuousHigherD} are reminiscent of the Borisov-Dudley-Durst condition \eqref{eq:BorisovDurst}. Indeed, they naturally appear by partitioning the Euclidean space $\R^d = \bigcup_{k \in \mathbb{Z}^d} [k,k+1)$ and controlling the empirical process indexed over the respective function class restricted to individual partitions \cite[Theorem 2.10.24]{van1996weak}. For the latter, the proof of \Cref{cor:CLT_EuclSpace} exploits well-known metric entropy bounds for the class of $\alpha$-H\"older and concave, Lipschitz functions, respectively. Notably, crucial to CLTs for dimension $d=2,3$ is condition \eqref{eq:locallySemiconcave} that enables suitable upper bounds for the metric entropy of $\FC_c$ and~$\FC_c\powerC$.

\begin{remark}[On the assumptions]\label{rem:AssumptionsDim23} A few words regarding the required assumptions for the CLTs of this subsection are in order. 
\begin{enumerate}
    \item[$(i)$] The partition of $\R^d = \bigcup_{k \in \mathbb{Z}^d} [k,k+1)$ by regular cubes is arbitrary and any partition of convex, bounded sets $I_k\subset\R^d$ with non-empty interior such that $\sup_{k} \text{diam}(I_k)<\infty$ serves to derive the same conclusion. 
    
    \item[$(ii)$]
    Condition \eqref{eq:locallySemiconcave} is fulfilled if the cost function is twice continuously differentiable in both components with a uniform bound $K>0$ on the Eigenvalues of its Hessian. In this setting the bound from \eqref{eq:locallySemiconcave} is valid for $\Lambda = K/2$. 
    \item[$(iii)$] The summability constraints \eqref{eq:BorisovDurstRealLine} and  \eqref{eq:BorisovDurstContinuousHigherD} are well-known in the context of empirical process theory \cite[Section 2.10.4]{van1996weak}. A sufficient condition is given in terms of finite moments $\EV{\norm{X}_\infty^{2d+\delta}}< \infty$ for some $\delta>0$ since
    \begin{equation*}
        \sum_{k \in \mathbb{Z}^d} \sqrt{\mu([k,k+1))} \lesssim 2^d \sum_{n = 1}^{\infty} \sqrt{n^{2d-2} \mathbb{P}(\norm{X}_\infty \geq n)} \leq 2^d \sqrt{\EV{\norm{X}_\infty^{2d+\delta}}} \sum_{n = 1}^{\infty} n^{-(1+\delta/2)},
    \end{equation*}
    where the latter inequality follows by Markov's inequality. Notably, for compactly supported measures the condition is vacuous.
\end{enumerate}
\end{remark}

\begin{example}[Costs $c(x,y) = h(x-y)$, $d = 2,3$]
Let us provide a few consequences of \Cref{cor:CLT_EuclSpace} by taking costs $c(x,y) = h(x-y)$ for some suitable function $h\colon \RR^d\rightarrow \RR_+$ with $d\in \{2,3\}$. If $h$ is a bounded and twice continuously differentiable function with uniformly bounded first and second derivatives, then the required conditions on the cost in \Cref{cor:CLT_EuclSpace} are fulfilled. Moreover, since the pointwise minimum of bounded, Lipschitz, semi-concave functions with bounded modulus also exhibits these properties, we find that \Cref{cor:CLT_EuclSpace} also covers thresholded costs $c_{p,T}(x,y) = \min(\norm{x-y}^p,T)$ for $p\geq 2$ and $T>0$. 
Costs for the canonical flat torus  $\tilde c(x,y) = \min_{z \in \mathbb{Z}^d} h(x-y-z)$, which have been considered by \cite{gonzalez2021two} for $h(x) = \norm{x}^2$, also fulfill these conditions. 
For these cases, \Cref{cor:CLT_EuclSpace} provides novel CLTs for $\mu, \nu \in \PC(\RR^d)$ if the summability constraint \eqref{eq:BorisovDurstContinuousHigherD} is satisfied. Moreover, if $\mu$ and $\nu$ are both compactly supported, condition \eqref{eq:BorisovDurstContinuousHigherD} is vacuous and it suffices that $h$ is twice continuously differentiable, hence, our theory encompasses $c(x,y) = \norm{x-y}^p$ with $p \geq 2$.

In view of \Cref{sec:DegenerateLimitLaws} the resulting limit distributions typically do not degenerate for $\mu \neq \nu$ since triviality of Kantorovich potentials is linked to the underlying geometry of the corresponding measures' supports (\Cref{thm:ConstantPotentials}) and appears to be limited to exotic settings. 
In contrast, under $\mu = \nu$ constant potentials do exist if $h(0) = 0$   (\Cref{cor:ConstantPotentials1}) and degeneracy of the distributional limits depends on whether the Kantorovich potentials are unique. Indeed, if $\supp(\mu)$ is the closure of a connected open set, then uniqueness holds \citep[Corollary 2]{staudt2021uniqueness} and the convergence rate of the empirical OT cost is strictly faster than $n^{-1/2}$. Notably, this in line with results by \cite{ajtai1984optimal} and \cite{ledoux2019OptimalMatchingI} stating that the uniform distribution $\mu = \Unif([0,1]^d)$ fulfills $\EV{\OT_p(\hat \mu_n, \nu)} = o(n^{-1/2})$ for $d \leq 3$ and $p \geq 2$. However, if $\supp(\mu)$ is disconnected and those components are cost-separated\footnote{For costs $c(x,y)\colon \XC\times \XC\rightarrow \R_+$ with $c(x,x)=0$ two subsets $\XC_1, \XC_2\subseteq \XC$ are said to be cost-separated if $\inf_{x\in \XC_1,y\in \XC_2}\min(c(x,y), c(y,x))>0$.}, then non-trivial Kantorovich potentials also exist \citep[Lemma 11]{staudt2021uniqueness} leading to non-degenerate limit laws. 
\end{example}

\subsection{The $d$-dimensional Euclidean Space for $d\geq 4$}\label{ex:Dgtr4}

	Beyond the low dimensional setting $d \leq 3$ treated so far, CLTs centered by the population quantity cannot hold in generality for $d \geq 5$ and remain a delicate issue for $ d = 4$. 
	For instance, under squared Euclidean costs $c(x,y) = \norm{x-y}^2$ and $\mu = \Unif([a,a+1])$, $\nu = \Unif([b,b+1])$ for $a,b \in \R^d$ 
	the OT plan between $\mu$ and $\nu$ is given by $(\id, \id +b-a)_{\#}\mu$ which implies according to  \citet[Theorem 6]{Manole2021_Plugin} that
$$
    \EV{\OT_2(\hat \mu_n, \nu)} - \OT_2(\mu, \nu) =   \EV{\OT_2(\hat \mu_n, \mu)}.
$$
	For $d \geq 3$ it is known by \cite{ ledoux2019OptimalMatchingI}  that $\EV{\OT_2(\hat \mu_n, \mu)} \asymp n^{-2/d}$. 
	Recalling the CLT by \cite{delbarrio2019} in \eqref{eq:delbarrioCLT} this implies the random sequence  
\begin{align*}
	&\sqrt{n}(\OT_2(\hat \mu_n, \nu) - \OT_2(\mu, \nu)) \\
	=& \sqrt{n}(\OT_2(\hat \mu_n, \nu) - \EV{\OT_2(\hat \mu_n, \nu)}) + \sqrt{n}\EV{\OT_2(\hat \mu_n, \mu)}
\end{align*}
 to be tight for $d = 4$ and to diverge almost surely to $\infty$ for $d \geq 5$. In conjunction with \citet[Theorem 6.2]{goldman2021convergence} who prove $\lim_{n \rightarrow \infty} \sqrt{n}\EV{\OT_2(\hat \mu_n, \mu)} = K$ under $d = 4$ for some positive constant $K>0$, it thus follows from \eqref{eq:delbarrioCLT} for $n \rightarrow \infty$ that
\begin{equation}\label{eq:CLTdim4}
   \sqrt{n}(\OT_2(\hat \mu_n, \nu) - \OT_2(\mu, \nu)) \konvD Z \sim \NC(K, \Var_{X\sim\mu}[f(X)]),
\end{equation}
where $f\in S_c(\mu, \nu)$ is a Kantorovich potential between $\mu$ and $\nu$. 
Notably, in this setting the Kantorovich potential is  unique (recall \Cref{def:UniquePotentials}) and also trivial if and only if $\mu = \nu$ (or equivalently if $a = b$).

In view of \Cref{thm:NormalLimits} the previous example also highlights that the function class $\FC_c$ is not $\mu$-Donsker for $d\geq 4$, since otherwise a centered tight normal limit would result. 
 In~particular, the CLT in \eqref{eq:CLTdim4} resembles the first asymptotic distributional limit for the empirical squared $2$-Wasserstein distance for $d = 4$ where the centering is given by the population quantity. 
 An exception concerning the Donsker property of $\FC_c$ in the high-dimensional regime occurs for different probability measures if one of them is supported on a sufficiently low dimensional space as will be detailed in the next subsection.

\subsection{Empirical Optimal Transport under Lower Complexity Adaptation} 

In this section, we highlight our CLTs for empirical OT in view of the recently discovered  \emph{lower complexity adaptation} principle \citep{hundrieser2021convergence}. It states that statistical rates to estimate the empirical OT cost between two different probability measures $\mu$ and $\nu$ are  driven by the less complex measure, e.g., the one with lower dimensional support. In light of this principle, we emphasize that our CLTs extend beyond the low-dimensional Euclidean case provided that at least one measure has some low-dimensional compact support. The main result relies on the observation that the uniform metric entropies for $\FC_c$ and $\FC_c\powerC$ coincide in such settings \cite[Lemma 2.1]{hundrieser2021convergence}. Hence, under suitable bounds on the uniform metric entropy for only one of the function classes $\FC_c$ or $\FC_c\powerC$, it follows that \emph{both} are universal Donsker (Definition \ref{def:Donsker}).

The following result makes use of this observation and is specifically tailored to settings where $\XC$ has low intrinsic dimension which allows the complexity of $\FC_c$ to be suitably controlled such that it is universal Donsker.
To formalize this, we consider the setting where $\XC$ is a finite set or a compact submanifold of $\RR^d$ (see \citealt{lee2013smooth} for comprehensive treatment) with sufficiently small intrinsic dimension. 
Notably, the first setting covers semi-discrete OT \citep{aurenhammer1998minkowski,merigot2011multiscale, hartmann2020semi}, i.e., where one of the probability measures is assumed to be finitely supported.
Then, under suitable assumptions on the cost function the uniform metric entropy of $\FC_c$ is suitably bounded which enables CLTs for the empirical OT cost. 

\begin{theorem}[Semi-discrete]\label{thm:Semidiscrete}
	 Let $\XC$ be a finite set equipped with discrete topology, let $\YC$ be Polish space and consider a bounded and continuous cost function $c\colon \XC\times \YC\rightarrow \RR_+$. 
	  Then, for arbitrary $\mu \in \PC(\XC), \nu \in \PC(\YC)$ the weak limits from \eqref{eq:GeneralCLT} hold.
\end{theorem}

\begin{remark}[Limit distribution for  semi-discrete OT]
We comment on structural properties of the limit distribution underlying \Cref{thm:Semidiscrete}. When the support of $\nu$ is connected (and the cost function is continuous), then Kantorovich potentials are unique  \cite[Example 3]{staudt2021uniqueness}. Hence, in this setting the corresponding weak limit is always centered normal. Even if the support of $\nu$ is not connected, under a suitable non-degeneracy condition of the OT plan Kantorovich potentials still remain  unique. Hence, under uniqueness of Kantorovich potentials, \Cref{thm:ConstantPotentials} serves as a sharp statement for the degeneracy of the limit distribution (\Cref{rem:UniquenessGivesSharpness}). In  particular, this shows  $S_c(\mu, \nu)$ to be trivial only under certain geometrical configurations, whereas $S_c^c(\mu, \nu)$ to be generically non-trivial.

Let us also point out that, parallel and independently to this work, \cite{del2022central} recently also obtained a CLT for the empirical OT cost in the semi-discrete framework for unbounded cost functions. In their setting, our triviality statements on Kantorovich potentials (\Cref{thm:ConstantPotentials}) and  degeneracy results for the corresponding limit laws also apply since for semi-discrete OT the Kantorovich potentials are always continuous (as a finite minimum over continuous functions).
\end{remark}

\begin{theorem}[Manifolds]\label{cor:CLT_LCA} 
Let $\XC$ be an $s$-dimensional smooth compact submanifold of $\RR^d$ with $s\leq \min(3,d)$, let $\YC \subseteq \RR^d$ be compact and consider a continuous cost function $c\colon \RR^d\times \RR^d\rightarrow \RR_+$. Suppose one of the following two settings.
\begin{enumerate}
	\item[$(i)$] \,\,$s=1$ and $c$ is locally $\alpha$-H\"older for $\alpha\in (1/2,1]$. 
	\item[$(ii)$] $s = 2$ or $s = 3$ and $c$ is twice continuously differentiable.
\end{enumerate}
    Then, for arbitrary $\mu \in \PC(\XC), \nu \in \PC(\YC)$ the weak limits from \eqref{eq:GeneralCLT} hold.
\end{theorem}

\begin{proof}[Proofs for Theorems \ref{thm:Semidiscrete} and \ref{cor:CLT_LCA}]
Note that Assumptions \ref{ass:CostsContinuousAndBounded} and \ref{ass:EquicontinuityCompact} are fulfilled for both settings. 
By Section 3.1, Lemma A.4 and Lemma A.3 in \cite{hundrieser2021convergence} for all three cases, respectively, it follows for $\epsilon>0$ sufficiently small that 
$$ 
\log \NC(\epsilon,\FC_{c},\Vert\cdot\Vert_\infty) \lesssim_{ \XC ,c}
\begin{cases}  \log\left(\left\lceil\epsilon^{-1}\right\rceil\right) & \text{ for \Cref{thm:Semidiscrete},}\\
 	\epsilon^{-s/\alpha} & \text{ for setting $(i)$ in \Cref{cor:CLT_LCA},}\\
	 \epsilon^{-s/2} & \text{ for setting $(ii)$ in \Cref{cor:CLT_LCA}.}
\end{cases}
$$
Moreover, by \citet[Lemma 2.1]{hundrieser2021convergence} it holds for any $\epsilon>0$ that 
$$\NC(\epsilon, \FC_c\powerC, \norm{\cdot}_\infty) = \NC(\epsilon, \FC_c\powerC, \norm{\cdot}_\infty),$$
 which implies that identical bounds on the uniform metric entropy of $\FC_c\powerC$ hold for all settings. Since the square root of the uniform metric entropy is for all settings integrable with respect to $\epsilon>0$ near zero, it follows by \citet[Theorem 2.5.2]{van1996weak} that both $\FC_c$ and $\FC_c\powerC$ are universal Donsker. We thus conclude from \Cref{thm:GeneralCLT} the CLTs for the empirical OT cost \eqref{eq:GeneralCLT} for arbitrary probability measures $\mu\in \PC(\XC),\nu\in \PC(\YC)$ .
\end{proof}

We emphasize that in contrast to previous CLTs from previous Subsections, no summability conditions are necessary in Theorems \ref{thm:Semidiscrete} and \ref{cor:CLT_LCA} for $\mu$ and $\nu$. 
Additionally, let us point out that \cite{hundrieser2021convergence} also provide uniform metric entropy bounds for $\FC_c$ under more general ground spaces $\XC$, such as metric spaces or parametrized surfaces with low intrinsic dimension, as well as $\alpha$-H\"older costs of smoothness degree $\alpha \in (1,2]$. As long as the integrability condition by \citet[Theorem 2.5.2]{van1996weak} for $\epsilon>0$ near zero is fulfilled, these bounds can also be employed for the derivation of CLTs of the empirical OT cost in even more general settings.

\begin{remark}[Wasserstein distance in high-dimensional spaces]\label{rem:lowerboundsandCLT}
As a consequence of Theorems \ref{thm:Semidiscrete} and \ref{cor:CLT_LCA}, we obtain CLTs for the empirical Wasserstein distance, i.e., for the Euclidean cost function $c(x,y) = \norm{x-y}^p$ for $p\geq 1$ even beyond $d \leq 3$ as long as both $\mu$ and $\nu$ have bounded support and one of them is sufficiently low dimensional. 
 More precisely, if $\mu$ is supported on a finite set or on an $s$-dimensional compact submanifold with $s=1$ and $p \geq 1$, or in case $s \in \{2,3\}$ and $p \geq 2$,
  our asymptotic results from~\eqref{eq:GeneralCLT} remain valid. Notably, for the latter case the asymptotic results still hold for any $p\in [1,2)$ if $\supp(\nu)$ is disjoint from $\supp(\mu)$. Indeed, in this setting one can extend the cost function $c|_{\Sigma}$ for $\Sigma=\supp(\mu)\times \supp(\nu)$ in a smooth manner to $\R^{2d}$, e.g., by the extension theorem of \cite{whitney1934analytic}, without altering the population and empirical OT cost. 
 
 In particular, under $s<d$ it follows by \Cref{cor:ConstantPotentials2} $(ii)$ that $S_c\powerC(\mu, \nu)$, i.e., the set of  Kantorovich potentials corresponding to $\nu$ is not trivial if $\supp(\nu)$ has non-empty interior. Hence, for $s$ sufficiently small when replacing the measure $\nu$ by its empirical measure $\hat \nu_m$ the limit laws do not degenerate. When instead replacing $\mu$ by $\hat \mu_n$, the limit law could indeed degenerate although this is rather the exception than the rule (see \Cref{thm:ConstantPotentials}). We recall \Cref{fig:ConstantPotentials}$(a)$, for an example where all Kantorovich potentials $S_c(\mu, \nu)$ are trivial, whereas $S_c\powerC(\mu, \nu)$ is not.
\end{remark}

\vspace{0.1cm}
\begin{ackno}
S. Hundrieser, M. Klatt, and T. Staudt gratefully acknowledge support from the DFG RTG 2088 and DFG EXC 2067/1- 390729940, and A. Munk of DFG CRC 1456. 
\end{ackno}

%\bibliographystyle{apalike2}
%\bibliography{biblio}

%\clearpage

\begin{appendix}

\section{Appendix}\label{sec:Proofs}

\subsection*{Omitted proofs for \Cref{sec:DegenerateLimitLaws}}

\setcounter{equation}{0}
\setcounter{theorem}{0}

\begin{proof}[Proof of \Cref{eq:ProjectedMeasureOT}]
We prove this equivalence under the assumption of a continuous cost function $c$ and a finite value $\OT_c(\mu, \nu) < \infty$. We set $h = \inf_{x'\in\supp(\mu)} c(x', \cdot)$, which
is upper semi-continuous as an infimum of continuous functions.

If $\mu\in P_c(\nu)$ with a corresponding coupling $\pi\in \Pi(\mu, \nu)$ in~\eqref{eq:ProjectedMeasure}, then it follows that
\begin{equation*}
    \OT_c(\mu, \nu)
    \le \int_{\XC\times\YC} c(x,y)\dif\pi(x, y)
    = \int_\YC h(y)\dif\nu(y)
    \le \OT_c(\mu, \nu),
\end{equation*}
where the final step relies on the fact that $h(y) \le c(x, y)$ for all $(x, y)\in\supp(\mu)\times\supp(\nu)$. In particular, $\pi$ is an OT plan.
Conversely, let the right hand side of \eqref{eq:ProjectedMeasureOT} be satisfied with an OT plan $\pi\in\Pi(\mu, \nu)$. Then
\begin{equation*}
    0
    \le
    \OT_c(\mu, \nu)
    =
    \int_{\XC\times\YC}c(x, y)\dif\pi(x, y)
    =
    \int_{\YC}h(y)\dif\pi(x, y) < \infty.
\end{equation*}
Since $h(y) \le c(x, y)$, this equation implies that equality of the integrands holds $\pi$-almost surely, and thus at least in a dense subset $A\subset\supp(\pi)$. For any $(x, y)\in\supp(\pi)$, choose a converging sequence $(x_n, y_n)_n\subset A$. Then, via continuity of $c$ and semi-continuity of $h$,
\begin{equation*}
  h(y) \le c(x, y)
  = \lim_{n\to\infty} c(x_n, y_n)
  = \lim_{n\to\infty} h(y_n)
  \le h(y),
\end{equation*}
establishing $h(y) = c(x, y)$ for all $(x, y)\in\supp(\pi)$. This implies $\mu\in P_c(\nu)$.
\end{proof}

\begin{proof}[Proof of \Cref{thm:ConstantPotentials}]
According to Lemma~3 in \cite{staudt2021uniqueness}, trivial Kantorovich potentials exist if and only if trivial Kantorovich potentials exist in the formally restricted problem where $\XC$ is replaced by the support of $\mu$ and $\YC$ is replaced by the support of $\nu$.
We can therefore assume $\XC = \supp(\mu)$ and $\YC = \supp(\nu)$.

Let $f$ be a trivial Kantorovich potential and $\pi\in\Pi(\mu, \nu)$ be an OT plan. By \Cref{lem:ContinuityOfKantPotentials}, the Kantorovich potential $f$ is continuous and we can assume that $f \equiv a$ on $\supp(\mu)$ for some $a\in\RR$. Then, for each $(x, y)\in\supp(\pi)$, it holds that $f^c(y) = c(x, y) - a$. Therefore,
\begin{equation*}
    c(x, y) = f^c(y) + a = \inf_{x'\in\supp(\mu)} c(x', y),
\end{equation*}
which establishes relation~\eqref{eq:ProjectedMeasure} and shows $\mu\in P_c(\nu)$.
Conversely, assume that \eqref{eq:ProjectedMeasure} holds for some $\pi\in\Pi(\mu, \nu)$.
Setting $f \equiv 0$, we observe that
\begin{equation*}
    f(x) + f^c(y) = \inf_{x'\in\supp(\mu)} c(x', y) = c(x, y)
    \qquad\text{for all}~(x,y)\in\supp(\pi).
\end{equation*}
According to standard OT theory \citep{santambrogio2015optimal}, this equality also holds if $f$ is replaced by the $c$-concave function $f^{cc} \ge f$, which implies that $f^{cc}\in S_c(\mu, \nu)$ is a Kantorovich potential (up to a suitable additive constant). Since $f^{cc}(x) = c(x, y) - f^c(y) = f(x) = 0$ on the set $\{x\,|\,(x,y)\in\supp(\pi)\}$, which has full $\mu$-measure, $f^{cc}$ is trivial.
\end{proof}

\begin{proof}[Proof of \Cref{cor:ConstantPotentials1}]
  Under condition $(i)$, we find that $\pi = (\id, \id)_{\#}\mu\in\Pi(\mu, \nu)$ satisfies \eqref{eq:ProjectedMeasure}, implying $\mu\in P_c(\nu)$. If condition $(ii)$ holds, the right hand side of \eqref{eq:ProjectedMeasureOT} can easily be established since $\pi = (p, \id)_{\#}\nu \in \Pi(\mu, \nu)$ is optimal. In both cases, the claim then follows by \Cref{thm:ConstantPotentials}.
\end{proof}

\begin{proof}[Proof of \Cref{cor:ConstantPotentials2}]
  Due to \Cref{thm:ConstantPotentials}, it is in all cases sufficient to show that $\mu\not\in P_c(\nu)$.
  Under condition $(i)$, we find $\OT_c(\mu, \nu) > 0$. At the same time,  $\supp(\nu)\subset\supp(\mu)$ implies that the equality on the right hand side of \eqref{eq:ProjectedMeasureOT} cannot hold. 

  We next note that $A_\pi \coloneqq \{x\,|\,(x, y)\in\supp(\pi)\}$ is dense in $\supp(\mu)$ for any $\pi\in\Pi(\mu, \nu)$.
  Under condition $(ii)$, the set $U \coloneqq \interior{\supp(\mu)} \setminus \supp(\nu)$ is non-empty and open, since $\supp(\nu)$ is closed.
  Therefore, we find $x\in U\cap A_\pi$ and $y\in\supp(\nu)$ with $(x,y)\in\supp(\pi)$. Let $u = x - y$. As $U$ is open, there exists $\epsilon > 0$ small enough that $x' \coloneqq x - \epsilon\,u\in U\subset\supp(\mu)$. Then, employing the strict monotonicity of~$h$, 
  \begin{equation*}
      c(x', y) = h(\|x' - y\|) = h\big((1-\epsilon)\|x-y\|\big) < h(\|x - y\|) = c(x,y)
  \end{equation*}
  for $(x,y)\in\supp(\pi)$ and $x' \in \supp(\mu)$, so $\mu\not\in P_c(\nu)$ follows by \Cref{def:projectedmeasure}.
  
  Under condition $(iii)$, let $U \coloneqq \supp(\mu) \setminus \Gamma_c(\mu, \nu)$. For any coupling $\pi\in\Pi(\mu, \nu)$, we use the density of $A_\pi$ in $\supp(\mu)$ to find $x\in U \cap A_\pi$, implying the existence of $y\in\supp(\nu)$ with $(x, y)\in\supp(\pi)$. We conclude $c(x, y) > \inf_{x'\in\supp(\mu)}c(x', y)$, since $x$ would otherwise be an element of $\Gamma_c(\mu, \nu)$. By \Cref{def:projectedmeasure}, $\mu\not\in P_c(\nu)$ follows.
 \end{proof}

\subsection*{Regularity of Kantorovich Potentials}

Assumptions imposed on the cost function $c\colon \XC\times \YC\to \mathbb{R}_+$ translate to properties on the function class $\FC_c$ and its $c$-conjugate $\FC_c^c$. A simple observation is uniform boundedness of $\FC_c$ if the cost is non-negative and bounded. Indeed, for any $f\in\FC_c$ it holds that
\begin{align*}
    -\norm{c}_\infty \leq \inf_{y\in\YC} c(x,y) -\norm{c}_\infty\leq f(x)\leq \inf_{y\in\YC} c(x,y) \leq \Vert c\Vert_\infty
\end{align*}
and analogously for any $f\in\FC_c^c$. We summarize additional findings in this regard in the following two statements (see also \cite{gangbo1996geometry,villani2008optimal,santambrogio2015optimal,staudt2021uniqueness} for further details).

\begin{lemma}[Continuity of Kantorovich potentials]\label{lem:ContinuityOfKantPotentials}
Consider Polish spaces $\XC$ and $\YC$ and a cost function $c\colon \XC\times \YC\rightarrow\R_+$ satisfying Assumption \ref{ass:CostsContinuousAndBounded} combined with \ref{ass:EquicontinuityLocallyCompact} or \ref{ass:EquicontinuityCompact}. Then, Kantorovich potentials $S_c(\mu,\nu)$ and $S_c\powerC(\mu, \nu)$ are continuous on the supports $\mu$ and $\nu$, respectively. 
\end{lemma}

\begin{proof}
Equicontinuity of the partially evaluated costs implies continuity of Kantorovich potentials since the modulus of continuity of a function class is preserved under pointwise infima or suprema (see \citealt[Section 1.2]{santambrogio2015optimal}, for details). This implies continuity of Kantorovich potentials on both $\XC$ and $\YC$ under Assumption \ref{ass:EquicontinuityLocallyCompact}, whereas under Assumption \ref{ass:EquicontinuityCompact} this only implies continuity on $\XC$. Moreover, under Assumption \ref{ass:EquicontinuityCompact} the space $\XC$ is compact, hence continuity of the costs asserts by \citet[Lemma 2]{staudt2021uniqueness} continuity of Kantorovich potentials on the support of $\nu$.
\end{proof}

As particular instances where structural properties of the cost function are inherited to  Kantorovich potentials, we focus on the setting of H\"older smoothness and semi-concavity.

\begin{lemma}\label{lem:OTpotentials}
Let $\XC$ and $\YC$ be normed spaces. 
\begin{itemize}
    \item[(i)] If for some $\alpha\in(0,1]$ and $L>0$ the cost $c(\cdot,y)$ is $(\alpha,L)$-H\"older as a function in $x$ for all $y\in\YC$, then any $f \in\FC_{c}$ is $(\alpha,L)$-H\"older continuous. 
    \item[(ii)] If for some $\lambda>0$ the cost $c(\cdot,y)$ is $\Lambda$-semi-concave as a function in $x$ for all $y\in\YC$, then any $f\in\FC_{c}$ is $\lambda$-semi-concave.
\end{itemize}
\end{lemma}
 
Note that reversing the roles of $x$ and $y$ leads to analogous results for the $c$-conjugate function class $\FC_c^c$ if $\{c(x,\cdot)\mid x\in \XC\}$ is uniformly H\"older smooth or semi-concave.

\begin{proof}
Suppose $c(\cdot,y)$ is $(\alpha,L)$-H\"older continuous for any $y\in\YC$. Then, it follows that
\begin{equation*}
   f(x)=\inf_{y^\prime\in\YC} c(x,y^\prime)-g(y^\prime) \leq c(x,y)-g(y)\leq c(x^\prime,y)-g(y)+L\Vert x-x^\prime\Vert^\alpha.
\end{equation*}
Taking the infimum on the right hand side with respect to $y$ yields
\begin{equation*}
    f(x)\leq f(x^\prime)+L\Vert x-x^\prime\Vert^\alpha.
\end{equation*}
Changing the roles of $x$ and $x^\prime$ proves that $\vert f(x)-f(x^\prime)\vert \leq L\Vert x-x\Vert^\alpha$ for any $f\in\FC_c$. Suppose that $c(\cdot,y)$ is $\Lambda$-semi-concave, i.e., $c(\cdot,y)-\Lambda\Vert \cdot \Vert^2$ is concave for all $y\in\YC$ as a function of $x$. It then follows that
\begin{align*}
    &f(tx+(1-t)x^\prime)-\Lambda\Vert tx+(1-t)x^\prime\Vert^2\\
    &\geq t\left(\inf_{y\in\YC} c(x,y)-\Lambda\Vert x \Vert^2-g(y)\right) + (1-t) \left( \inf_{y\in\YC} c(x^\prime,y)-\Lambda\Vert x^\prime \Vert^2-g(y) \right)\\
    &=t\left(f(x)-\lambda\Vert x\Vert^2\right) +  (1-t) \left( f(x^\prime)-\Lambda\Vert x^\prime\Vert^2\right).
\end{align*}
Hence, any $f\in\FC_c$ is itself a $\Lambda$-semi-concave function.
\end{proof}

\end{appendix}

\end{document}